\documentclass[11pt]{amsart}
\usepackage[margin=1in,letterpaper,portrait]{geometry}
\usepackage{latexsym,amssymb,verbatim,multirow,graphicx,epsfig,enumerate, paralist}
\usepackage[style=alphabetic]{biblatex}
\usepackage[T1]{fontenc} 
\bibliography{ref-fac-gen}
\usepackage{graphbox}
\usepackage{subcaption}
\usepackage{xcolor}
\usepackage[normalem]{ulem}
\usepackage{tikz}
\usepackage{bm}
\usepackage{mathabx}
\usepackage{float}
\usepackage{multirow}
\usepackage{thm-restate}
\usepackage{hyperref}

\theoremstyle{plain}
   \newtheorem{theorem}{Theorem}[section]
   \newtheorem{proposition}[theorem]{Proposition}
   
   \newtheorem{lemma}[theorem]{Lemma}
   \newtheorem{corollary}[theorem]{Corollary}

   \newtheorem*{theorem*}{Theorem}
\theoremstyle{definition}
   
   \newtheorem{definition}[theorem]{Definition}
   \newtheorem{example}[theorem]{Example}
   
   \newtheorem{question}[theorem]{Question}
   \newtheorem{remark}[theorem]{Remark}

\numberwithin{equation}{section}

\usepackage{xparse}

\DeclareDocumentCommand \ltr { o } {%
  \IfNoValueTF {#1} {%
    \ell_W^{\mathrm{full}} %
  }{%
    \ell_{#1}^{\mathrm{full}}%
  }%
}
\DeclareDocumentCommand \lR { o } {%
  \IfNoValueTF {#1} {%
    \ell_W^{\mathrm{red}} %
  }{%
    \ell_{#1}^{\mathrm{red}}%
  }%
}

\DeclareDocumentCommand \Fred { o } {%
  \IfNoValueTF {#1} {%
    F_W^{\mathrm{red}} %
  }{%
    F_{#1}^{\mathrm{red}}%
  }%
}

\newcommand{\op}[1]{\operatorname{#1}}

\newcommand{\Ftr}{F^\mathrm{full}}
\newcommand{\FFFtr}{\FFF^{\mathrm{full}}}
\newcommand{\GD}{\op{GD}}
\newcommand{\RGS}{\op{RGS}}

\newcommand\Symm{\mathfrak{S}}

\newcommand{\ttt}{\bm{t}}

\newcommand\rank{\operatorname{rank}}
\newcommand\codim{\operatorname{codim}}

\newcommand\wt{\col}
\newcommand{\col}{\operatorname{col}}

\newcommand{\FFF}{\mathcal{F}}
\newcommand{\RRR}{\mathcal{R}}

\newcommand{\defn}[1]{{\color{blue} \it {#1}}}

\newcommand{\BBB}{\mathcal{B}}
\newcommand{\III}{\mathcal{I}}

\newcommand{\AAA}{\mathcal{A}}

\newcommand\CC{{\mathbb{C}}}
\newcommand\ZZ{{\mathbb{Z}}}
\newcommand\NN{{\mathbb{N}}}
\newcommand\QQ{{\mathbb{Q}}}

\newcommand\GL{{\mathrm{GL}}}

\newcommand{\id}{\op{id}}

\begin{document}

\title[W-Hurwitz numbers: Part~III]{Hurwitz numbers for reflection groups III:\\ Uniform formulas}
\author{Theo Douvropoulos, Joel Brewster Lewis, Alejandro H. Morales}

\maketitle

\begin{abstract}
We give uniform formulas for the number of full reflection factorizations of a parabolic quasi-Coxeter element in a Weyl group or complex reflection group, generalizing the formula for the genus-$0$  Hurwitz numbers. This paper is the culmination of a series of three.
\end{abstract}

\section{Introduction}

In the late 19th century, before Poincar{\'e}'s \emph{Analysis Situs} and Major MacMahon's \emph{Combinatory Analysis}, Hurwitz \cite{Hurwitz} 
was the first to recognize that the structure of Riemann surfaces is intrinsically combinatorial.  He showed in particular that such surfaces with finitely many branch points can be encoded  by factorizations of elements in the symmetric group $\Symm_n$.
Hurwitz became interested in enumerating the different classes of Riemann surfaces and gave a complete answer for the case of genus-$0$ surfaces with all but one branch point being simple. In combinatorial terms, he showed this to be equivalent to counting minimum-length \emph{transitive} factorizations $t_1\cdots t_k=\sigma$ of a given element $\sigma$ in $\Symm_n$ as a product of \emph{transpositions} $t_i$, where transitivity refers to the natural action of the group $\langle t_1, \ldots, t_k \rangle$ generated by the factors on the set $\{1,\dots,n\}$. Hurwitz then gave a sketch of an inductive argument, reproduced in detail in \cite{strehl}, for the following remarkable product formula.

\begin{theorem}[{Hurwitz formula \cite{Hurwitz}}]
\label{thm:S_n genus 0}
  The minimum length of a transitive
  transposition factorization in $\Symm_n$ of a permutation of cycle type $\lambda:=(\lambda_1,\ldots,\lambda_r)$ is $n+r-2$. The number of such factorizations is
\begin{equation} \label{EQ: Hurwitz formula}
H_0(\lambda) =  (n+r-2)!\cdot n^{r-3}\cdot  \prod_{i=1}^r \frac{\lambda_i^{\lambda_i}}{(\lambda_i-1)!}.
\end{equation}
In particular, $H_0(1^n) = (2n-2)!\cdot n^{n-3}$ and $H_0(n) = n^{n-2}$.
\end{theorem}

The formula in Theorem~\ref{thm:S_n genus 0} is for what are now called the {\em (single) Hurwitz numbers of genus $0$}. These numbers also count certain connected graphs embedded in the sphere (the {\em planar maps}). In general, the \defn{genus-$g$ Hurwitz number $H_g(\lambda)$} counts transitive factorizations 
of a given element $\sigma$ in $\Symm_n$ of cycle type $\lambda$ into $k=n+r+2g-2$ transpositions, and also connected graphs embedded on orientable surfaces of genus $g$ (see, e.g., \cite{LandoZvonkin, CavalieriMiles,GJSurvey,ALCO_GJSurvey}).

In the 1980s, work of Stanley \cite{St80}, Jackson \cite{J88}, and others rekindled interest in the enumeration of factorizations in $\Symm_n$ even though, at the time, they were unaware of the topological context. Independently, the next few decades saw the emergence of \emph{Coxeter combinatorics}; one of its main breakthroughs was the realization that theorems about $\Symm_n$ are often shadows of more general results that hold for all reflection groups. In our context of factorizations, this means replacing \emph{transpositions} in $\Symm_n$ with \emph{reflections} in a reflection group $W$, as in \cite{chapuy_Stump}.

The intersection of these two areas has witnessed a lot of research activity recently (see, e.g., \cite{chapuy_Stump,LM,D2}), especially for factorizations of \emph{Coxeter elements} in $W$, generalizing the case of a long cycle $\lambda=(n)$ in $\Symm_n$. In general, however, analogs of Theorem~\ref{thm:S_n genus 0} have been hard to find, not least because it is unclear how to define transitivity in reflection groups.  Several recent papers \cite{BGJ, LM, PR} have explored the concept in the infinite family of complex reflection groups, exploiting the permutation action on coordinate axes, but there are no similar structures in the other (exceptional, primitive) cases.  In $\Symm_n$, transitivity corresponds to the connectedness of the associated maps or Riemann surfaces, but neither of these have analogs for general reflection groups.

An equivalent way to interpret the notion of transitivity in $\Symm_n$ is to require that the factorization cannot be realized in any \emph{proper} \defn{Young subgroup} (a subgroup generated by transpositions), or in other words that the factors generate the full group $\Symm_n$.  This interpretation makes sense for an arbitrary reflection group $W$, where we will thus say that $t_1\cdots t_k=g$ is a \defn{full reflection factorization} of an element $g\in W$ if the factors $t_i$ are reflections and they generate the \emph{full} group $W$. We are particularly interested in the ``genus-$0$'' case, where the number $k$ of factors is minimum for the element $g$. We call this number $k$ the \defn{full reflection length} of $g$ and denote it by \defn{$\ltr(g)$}, leaving the symbol \defn{$\lR(g)$} to stand  for the usual (reduced) reflection length of $g$, which does not require fullness of the factorization. We write \defn{$\RRR$} for the set of reflections in $W$ and \defn{$\Ftr_W(g)$} for the number of minimum-length full reflection factorizations of $g$, i.e.,
\[
\Ftr_W(g):=\#\left\{(t_1,\ldots,t_k)\in\RRR^k:\ t_1\cdots t_k=g,\ \langle t_1,\ldots, t_k\rangle=W,\ \text{and }k=\ltr(g) \right\}.
\] 

In this work we establish uniform product formulas for the counts $\Ftr_W(g)$, which one might call the {\em $W$-Hurwitz numbers}. They generalize Hurwitz's formula of Theorem~\ref{thm:S_n genus 0} to well generated complex reflection groups $W$ for a wide class of elements $g\in W$, known as \emph{parabolic quasi-Coxeter} elements. This class contains all parabolic Coxeter elements, and thus (in the case $W = \Symm_n$) all elements of the symmetric group.  In the next Section~\ref{intro: main}, we present our main enumerative results separately for Weyl groups (Theorem~\ref{Thm: Weyl case}) and well generated complex reflection groups (Theorem~\ref{thm:main}).  
We give in Section~\ref{intro: recovering hurwitz} a short demonstration of the concordance between our main theorems and Hurwitz's Theorem~\ref{thm:S_n genus 0} in the case of the symmetric group $\Symm_n$.  Then we end this introduction in Section~\ref{sec:overview} with an overview of the rest of the paper.

\subsection{Main theorems}
\label{intro: main}

There are many combinatorial approaches to the proof and interpretation of Hurwitz's formula (e.g., \cite{strehl,B-MS,GJ99b,DPS,DPS_II}; see also the account in \cite{CavalieriMiles}). The ones that are most relevant to our work relate transitive factorizations with tree-like structures. For instance, as one special case of Theorem~\ref{thm:S_n genus 0} we have that $H_0(n) = n^{n - 2}$ is the number of trees on $n$ labeled vertices, and an elegant combinatorial proof of this may be found in \cite{Denes}.  In \cite{DPS_II}, Duchi--Poulalhon--Schaeffer gave a bijective proof of the full Theorem~\ref{thm:S_n genus 0} in which the term $n^{r-3}$ roughly counts certain trees whose vertices are the $r$-many cycles of $g\in\Symm_n$, which one may call \emph{relative trees} on the cycles of $g$. In the setting of general reflection groups, we find that trees are replaced in this central role by a natural structure we call \emph{relative generating sets} (defined in Section~\ref{sec:rgs}).

Not all elements in a reflection group admit relative generating sets.  The ones that do form a wide class of elements called \emph{parabolic quasi-Coxeter elements} (defined in Section~\ref{sec:parabolic} below); in particular, in $\Symm_n$, all elements are parabolic quasi-Coxeter.  It is these elements to which our main Theorems~\ref{Thm: Weyl case} and~\ref{thm:main} apply.

Every element $g$ in a reflection group $W$ has a decomposition $g = g_1 \cdots g_r$ given by the decomposition of its \defn{parabolic closure} $W_g$ (the smallest parabolic subgroup that contains $g$) into irreducible factors $W_g = W_1 \times \cdots \times W_r$.  For parabolic quasi-Coxeter elements, this \defn{generalized cycle decomposition} satisfies a further uniqueness property, extending the usual cycle decomposition of permutations---see Section~\ref{sec:parabolic}.
These objects underlie our generalization of Theorem~\ref{thm:S_n genus 0}.

\subsubsection*{The Weyl group case} 
In the case of Weyl groups, the counts $\Ftr_W(g)$ are given by a very appealing product formula where the number $\#\RGS(W,g)$ of relative generating sets appears as a direct factor. The \defn{connection index  $I(W)$} of a Weyl group $W$ is defined as the index of the root lattice in the weight lattice of $W$ (see Section~\ref{Sec: real refl grps}).

\begin{restatable*}[main theorem for Weyl groups]{theorem}{weyltheorem}
\label{Thm: Weyl case}
For any Weyl group $W$ and any parabolic quasi-Coxeter element $g\in W$
with generalized cycle decomposition $g=g_1\cdot g_2\cdots g_r$, we have
\begin{equation}
\label{Eq: weyl thm}
\Ftr_W(g)=\ltr(g)!\cdot\prod_{i=1}^r\dfrac{\Fred(g_i)}{\lR(g_i)!}\cdot 
\#\RGS(W,g)\cdot\dfrac{I(W_g)}{I(W)}
\, ,
\end{equation}
where $W_g$ is the smallest parabolic subgroup containing $g$ and $I(W)$ is the connection index of $W$. In particular, if $g$ is the identity element and $n$ is the rank of $W$, we have
\begin{equation}
\label{eq: intro weyl id case}
\Ftr_W(\id)=(2n)!\cdot\#\RGS(W)\cdot\dfrac{1}{I(W)}.
\end{equation}
\end{restatable*}

For a direct comparison with Theorem~\ref{thm:S_n genus 0}, see Section~\ref{intro: recovering hurwitz} below.

Theorem~\ref{Thm: Weyl case} has a particularly attractive form for parabolic Coxeter elements, where the product structure and analogy to Theorem~\ref{thm:S_n genus 0} is even more apparent. 

\begin{restatable*}{corollary}{weylthmdecompirred}
With the notation of Theorem~\ref{Thm: Weyl case}, we further have that if $g$ is a parabolic Coxeter element and $W_g=W_1\times\cdots \times W_r$ is the decomposition of $W_g$ into irreducibles, then 
\begin{equation}
\label{eq: intro weyl Coxeter}
\Ftr_W(g)=\ltr(g)!\cdot\#\RGS(W,g)\cdot\dfrac{I(W_g)}{I(W)}\cdot\prod_{i=1}^r\dfrac{h_i^{n_i}}{\#W_i},
\end{equation}
where $h_i:=|g_i|$ is the Coxeter number and $n_i$ the rank of $W_i$. 
 \end{restatable*}

\subsubsection*{The complex case}
Theorem~\ref{Thm: Weyl case} can be extended naturally to all \emph{well generated complex reflection groups} (finite subgroups of $\GL(\CC^n)$ generated by $n$ unitary reflections).  In this setting, the quantity $\#\RGS(W,g)$ must be replaced by the sum over the set $\RGS(W, g)$ of a certain statistic that in the case of Weyl groups always equals $I(W_g)/I(W)$. 
This statistic is a \emph{Grammian determinant} $\GD$ that is computed from sets $\bm{\rho}_g$ and $\bm\rho_{\ttt}$ of \emph{roots} associated to $g$ and to each relative generating set $\ttt \in\RGS(W,g)$.  (See Section~\ref{sec:Root systems} for the definitions of root system and roots for a complex reflection group, and Definition~\ref{Defn: Gram Det} for the definition of the Grammian determinant.)

\begin{restatable*}[main theorem for complex reflection groups]{theorem}{volumetheorem} 
\label{volume theorem}
\label{thm:main}
If $W$ is a well generated complex reflection group and $g$ is a parabolic quasi-Coxeter element in $W$ with generalized cycle decomposition $g = g_1 \cdots g_r$, then
\begin{equation}
\label{eq:volume theorem}
\Ftr_W(g) = \ltr(g)! \cdot \prod_{i = 1}^r \frac{\Fred(g_i)}{\lR(g_i)!} \cdot \sum_{ \ttt \in \RGS(W, g)} \frac{ \GD(\bm{\rho}_g)}{\GD(\bm{\rho_{\ttt}}\cup\bm{\rho}_g)},
\end{equation}
where $\bm\rho_{\ttt}$ denotes the set of roots associated with the relative generating set $\ttt$ and ${\bm\rho}_g$ denotes the set of roots associated with a fixed reduced reflection factorization of $g$.
\end{restatable*}

\subsection{Recovering the original Hurwitz formula for the symmetric group}
\label{intro: recovering hurwitz}

When $W = \Symm_n$ and $g \in W$ has cycle type $\lambda = (\lambda_1, \ldots, \lambda_r)$, the data in Theorem~\ref{Thm: Weyl case} are as follows: 
the full reflection length is $\ltr[\Symm_n](g) = n + r - 2$ \cite[\S6]{GJSurvey};
the generalized cycles of $g$ are precisely the cycles of $g$ in the usual sense (see Section~\ref{sec:parabolic});
for the cycle $g_i$ of length $\lambda_i$, we have $\lR[\Symm_n](g_i) = \lambda_i - 1$ and $\Fred[\Symm_n](g_i) = H_0(\lambda_i)= \lambda_i^{\lambda_i - 2}$ \cite{Denes};
the parabolic closure $W_g \cong \Symm_{\lambda_1} \times \cdots \times \Symm_{\lambda_r}$ is the subgroup that permutes the entries of each cycle among themselves;
$I(W) = n$ and $I(W_g) = \lambda_1 \cdots \lambda_r$ \cite[\S9.4]{Kane};
and the relative generating sets of $w$ are sets of $n - r - 1$ transpositions that form a tree when each of the $r$ cycles is collapsed to a single vertex, and there are $\lambda_1 \cdots \lambda_r \cdot n^{r - 2}$ of them \cite[Prop.~7.2]{DLM2}.
Substituting these values in to \eqref{Eq: weyl thm}, we have that
\begin{equation} \label{eq:our version Hurwitz Sn}
\Ftr_{\Symm_n}(g) = (n + r - 2)! \cdot \prod_{i = 1}^r \frac{\lambda_i^{\lambda_i - 2}}{(\lambda_i - 1)!} \cdot \lambda_1 \cdots \lambda_r \cdot n^{r - 2} \cdot \frac{\lambda_1 \cdots \lambda_r}{n}.
\end{equation}
Rearranging the powers of $\lambda_i$ and $n$, we immediately recover Theorem~\ref{thm:S_n genus 0}. 

Comparing the expressions \eqref{EQ: Hurwitz formula} and \eqref{eq:our version Hurwitz Sn} for the genus-$0$ Hurwitz numbers, we see a rearrangement of the powers of $\lambda_i$ that gives a new approach on the formula. We hope that with this paper we provide evidence that this perspective and the associated combinatorial structures reveal new, non-trivial properties of the Hurwitz formula.

\subsection{Overview of the paper} \label{sec:overview}

In Section~\ref{sec: background}, we introduce the necessary terminology and background on reflection groups, including a summary of the key results from the earlier papers \cite{DLM1, DLM2} in this series.  Section~\ref{Sec: proof main} is devoted to the case-by-case proof of our main results. For the groups in the infinite families, we compute the two sides of \eqref{eq:volume theorem} explicitly by combinatorial reasoning (representing both factorizations and relative generating sets by graph-theoretic objects).  For the exceptional Weyl groups, we develop a recurrence relation inspired by the cut-and-join equations of Goulden--Jackson \cite{GJ97} for $\Symm_n$ and rely on a large computer calculation to inductively prove \eqref{Eq: weyl thm}. For the exceptional complex (non-Weyl) groups, we use a different computer calculation involving character trace formulas and the Frobenius lemma to compute both the left side (as in \cite{DLM1}) and right side of \eqref{eq:volume theorem}. Finally, in Section~\ref{sec: further}, we make some remarks on the proof of the main results, and pose some open questions.

\section{Preliminaries and background}
\label{sec: background}

We begin with a brief overview of the machinery of reflection groups and a summary of results from the papers \cite{DLM1, DLM2} that are needed in this paper.  For thorough treatment of this background, see  \cite{Humphreys,Kane,LehrerTaylor, broue-book}.

\subsection{Real reflection groups}
\label{Sec: real refl grps}

Given a finite-dimensional real inner product space $V$, a \defn{reflection} $t$ is an orthogonal map whose \defn{fixed space} $\defn{V^{t}} := \{v \in V \colon t(v) = v\}$ is a hyperplane, that is, $\codim(V^{t})=1$. Equivalently, an orthogonal map is a reflection if it is diagonalizable, with one eigenvalue equal to $-1$ and all others equal to $1$.  A finite subgroup $W\leq \GL(V)$ is called a \defn{real reflection group} if it is generated by reflections. 

The real reflection groups are precisely the finite \defn{Coxeter groups}, those generated by a set $S = \{s_1, \ldots, s_n\}$ of reflections subject to relations $s_i^2 = 1$ and $(s_i s_j)^{m_{ij}} = 1$ for some integers $m_{ij}>1$. 
The elements of $S$ are called \defn{simple reflections} for $W$, and
the size $n$ of $S$ is called the \defn{rank} of $W$.

A real reflection group $W$ is completely determined by its set \defn{$\RRR$} of reflections.  Alternatively, it is determined by the associated \defn{reflection arrangement}, namely, the collection of fixed hyperplanes of the reflections in $\RRR$. 
For any reflection $t\in \RRR$, we may choose two vectors \defn{$\rho_t$} and \defn{$\widecheck{\rho}_t$}, both orthogonal to the fixed hyperplane $V^t$, that satisfy $\langle \rho_t,\widecheck{\rho}_t\rangle=2$.  They are known as the \defn{root} and \defn{coroot} associated to the reflection $t$, which they determine via the relation
\begin{equation}
t(v)=v-\langle v,\widecheck{\rho}_t\rangle\cdot \rho_t\label{eq: defn refln real}
\end{equation}
for all $v$ in $V$.  The collection $\defn{\mathrm{\Phi}}:=\{\pm \rho_t \colon t\in\RRR\}$ of roots and their negatives is the \defn{root system} of $W$; we choose the lengths of the roots $\rho_t$ so that $\Phi$ is $W$-invariant.  The \defn{simple (co)roots} $\defn{\alpha_i} := \rho_{s_i}$ and \defn{$\widecheck{\alpha}_i$}$\, := \widecheck{\rho}_{s_i}$ are the (co)roots associated to the simple reflections $s_i$.

For some real reflection groups, the lengths of the roots in a root system may be chosen so that the $\ZZ$-span of $\Phi$ is a lattice, called the \defn{root lattice} \defn{$\mathcal{Q}$}.  In this case, we say that $W$ is a \defn{Weyl group}.  For a Weyl group $W$, the coroots also span a lattice, called the \defn{coroot lattice} \defn{$\widecheck{\mathcal{Q}}$}.  
For any Weyl group, there is a natural inclusion of the root lattice $\mathcal{Q}$ inside the dual lattice $\mathcal{P}$ of the coroot lattice $\widecheck{\mathcal{Q}}$.  (The lattice $\mathcal{P}$, which will not play a major role in this paper, is called the \emph{weight lattice}.)  The \defn{connection index} $\defn{I(W)}:=[\mathcal{P}:\mathcal{Q}]$ is an important invariant of $W$.

\subsection{Complex reflection groups}

Let $V$ be a finite-dimensional complex vector space with a fixed Hermitian inner product.  A \defn{(unitary) reflection} $ t$ on $V$ is a unitary map whose fixed space $V^{t}$ is a hyperplane, that is, $\codim(V^{ t})=1$.  A finite subgroup $W\leq \GL(V)$ is called a \defn{complex reflection group} if it is generated by reflections. By extending scalars, every real reflection group may be viewed as a complex reflection group.
We say that a complex reflection group $W$ is \defn{irreducible} if there is no nontrivial subspace of $V$ stabilized by its action.  Shephard and Todd \cite{ShephardTodd} classified the complex reflection groups, as follows: every complex reflection group is a product of irreducibles, and every irreducible either belongs to an infinite three-parameter family $G(m,p,n)$, described below, or is one of 34 exceptional cases, numbered $G_4$ to $G_{37}$.  

If $W$ is a complex reflection group acting on $V$, the \defn{rank} of $W$ is the codimension $\codim(V^W)$ of the fixed space $\defn{V^W} := \bigcap_{w \in W} V^w$ of $W$ (or, equivalently, the dimension of the orthogonal complement $(V^W)^\perp$).  Every irreducible complex reflection group of rank $n$ can be generated by either $n$ or $n+1$ reflections. The groups in the first category are called \defn{well generated}; they are better understood and share many properties with the subclass of real reflection groups. 

\subsubsection*{The combinatorial family}

Let $m$, $p$, and $n$ be positive integers such that $p$ divides $m$.  Such a triple indexes a member $G(m, p, n)$ of the infinite family of complex reflection groups, which may be concretely described as 
\[
\defn{G(m, p, n)} := \left\{ \begin{array}{c} n \times n \textrm{ monomial matrices whose nonzero entries are} \\
\textrm{$m$th roots of unity with product an $\frac{m}{p}$th root of unity} \end{array} \right\}.
\]
The groups $G(m, 1, n)$ and $G(m, m, n)$ are well generated, but the intermediate groups $G(m, p, n)$ for $1 < p < m$ are not \cite[\S2.7]{LehrerTaylor}.

It is natural to represent such groups combinatorially.
We may encode each element $w$ of $G(m, 1, n)$ by a pair $[u; a]$ with $u \in \Symm_n$ and $a = (a_1, \ldots, a_n) \in (\ZZ/m\ZZ)^n$, as follows: for $k = 1, \ldots, n$, the nonzero entry in column $k$ of $w$ is in row $u(k)$, and the value of the entry is $\exp(2\pi i a_k/m)$.  With this encoding, it's easy to check that
\[
[u; a] \cdot [v; b] = [uv; v(a) + b], \quad \textrm{ where } \quad v(a) := \left( a_{v(1)}, \ldots, a_{v(n)} \right).
\]
This shows that the group $G(m, 1, n)$ is isomorphic to the \defn{wreath product} $\ZZ/m\ZZ \wr \Symm_n$ of a cyclic group with the symmetric group $\Symm_n$.

The reflections in $G(m, p, n)$ come in two families: for $1 \leq i < j \leq n$ and $k \in \ZZ/m\ZZ$ there is the \defn{transposition-like reflection}
\[
\defn{[(i j); k]} := [(i j); k e_i - k e_j] = [(i j); (0, \ldots, 0, k, 0, \ldots, 0, -k, 0, \ldots, 0)],
\]
which fixes the hyperplane $x_j = \exp(2\pi i k/m) x_i$, and if $p < m$ then for $i = 1, \ldots, n$ and $k = 1, \ldots, \frac{m}{p} - 1$ there is the \defn{diagonal reflection}
\[
[\id; kp e_i] = [\id; (0, \ldots, 0, kp, 0, \ldots, 0)],
\]
which fixes the hyperplane $x_i = 0$.  Observe that with this notation, $[(i j); k] = [(j i); -k]$.

Given an element $w = [u; a] \in G(m, p, n)$ and a subset $S \subseteq \{1, \ldots, n\}$, we say that $\sum_{k \in S} a_k$ is the \defn{color} of $S$; this notion will come up particularly when the elements of $S$ form a cycle in $u$.  When $S = \{1, \ldots, n\}$, we call $a_1 + \ldots + a_n$ the color of the element $w$, and we denote it $\defn{\wt(w)}$.  In this terminology, $G(m, p, n)$ is the subgroup of $G(m, 1, n)$ containing exactly those elements whose color is a multiple of $p$, and $\wt$ is a surjective group homomorphism from $G(m, p, n)$ to $p\ZZ/m\ZZ \cong G(m, p, 1) \cong \ZZ / (m/p)\ZZ$.  Meanwhile the map $[u; a] \mapsto u$ that sends an element of $G(m, p, n)$ to its \defn{underlying permutation} is a surjective group homomorphism onto $\Symm_n$.  More generally, for any $r \mid m$, there is a surjective homomorphism $\defn{\pi_{m/r}} : G(m, 1, n) \to G(r, 1, n)$ defined by $\pi_{m / r}([u; a]) = \left[u;  \frac{m}{r} \cdot a \right]$ (see \cite[Def.~2.1]{DLM1}).

\subsection{Root systems}
\label{sec:Root systems}

For a real reflection group, the root system carries essentially the same information as the reflection arrangement of $W$. The situation for complex groups is more complicated, beginning with the fact that there are infinitely many unit vectors orthogonal to each hyperplane (not just two, as in \S\ref{Sec: real refl grps}. It is a difficult problem to choose a collection of such orthogonal vectors for each hyperplane so that the whole family of roots is closed under multiplication by $W$, and even more difficult to define an object analogous to the root lattice of Weyl groups. Although these problems have been (at least partially) resolved (see Section~\ref{sec: michel root systems}), we will not make use of such stronger structures. For us, the root system will encode the reflection arrangement and we will further use a collection of normalizing coroots to keep track of the spectrum of each reflection.

To construct our roots and coroots, we pick, for each reflection $t$ of $W$ and its unique non-$1$ eigenvalue $\xi$, \emph{any} nonzero element $\rho_{t}$ of the $\xi$-eigenspace of $t$ and we call it the \defn{root associated to $t$}. Then we define the \defn{coroot associated to $t$} as the \emph{unique} vector $\widecheck{\rho}_{t}$ parallel to $\rho_{t}$ that satisfies $\langle \rho_{t}, \widecheck{\rho}_{t}\rangle = 1-\xi$ for the standard Hermitian inner product $\langle\cdot,\cdot\rangle$. In particular, this allows us to write the reflection $ t\in \GL(V)$ in the following standard form:
\[
 t(v)=v-\langle v,\widecheck{\rho}_{ t}\rangle\cdot \rho_{ t}.
\]

\subsection{Reflection length and full reflection length}
\label{sec:reflen}

The \defn{reflection length $\lR(w)$} of an element $w$ in a (real or complex) reflection group $W$ is the minimum length $k$ such that $w$ can be factored as $w = t_1 \cdots t_k$ for reflections $t_1, \ldots, t_k$.  The reflection length determines a partial order $\defn{\leq_{\RRR}}$ (the \defn{absolute order}) on the elements of $W$ via
\(
u \leq_{\RRR} v 
\) if and only if \(\lR(u)+\lR(u^{-1}v)=\lR(v)
\).

In the case of the symmetric group $W = \Symm_n$, reflection length has a simple combinatorial formula: $\lR(w) = n - c(w)$ where $c(w)$ is the number of cycles of $w$.  If $W$ is a real reflection group or the group $G(m, 1, n)$, reflection length has a geometric interpretation: if $W$ is rank $n$ then $\lR(w) = n - \dim(V^w)$ is the codimension of the fixed space of $w$. 
In the infinite family $W=G(m,p,n)$, one may also give combinatorial formulas for reflection length \cite[Thm.~4.4]{Shi2007}.

In \cite{DLM1}, we defined a reflection factorization $w = t_1 \cdots t_k$ of $w$ to be \defn{full} if the group $\langle t_1, \ldots, t_k\rangle$ generated by the factors is equal to the full group $W$, and we defined the \defn{full reflection length $\ltr(g)$} of $g$ to be the minimum length of a full reflection factorization:
\[
\ltr(g) := \min \Big\{ k : \exists\ t_1, \ldots, t_k \in \RRR \text{ such that } t_1 \cdots t_k = g \text{ and } \langle t_1, \ldots, t_k \rangle = W \Big\}.
\]
The next result gives a bound on the full reflection length of an element in terms of its reflection length.

\begin{corollary}[{\cite[Cor.~5.4]{DLM2}}]
\label{Corol: lr(g)+ltr(g)>=2m}
If $W$ is a complex reflection group and $r$ denotes the minimum size of a generating set of reflections for $W$, then for any element $g\in W$, we have 
\[
\lR(g)+\ltr(g)\geq 2r.
\]
\end{corollary}

We are primarily concerned with the \emph{enumeration} of full factorizations.  Given $g$ in $W$, we write $\defn{\Fred(g)}$ for the number of reduced reflection factorizations of $g$ and $\defn{\Ftr_W(g)}$ for the number of minimum-length full reflection factorizations of $g$.  
For $W$ in the infinite family, the values $\Ftr_W(g)$ were computed in \cite[Thm.~5.2]{DLM1} for arbitrary elements $g$; we record here the cases that will be necessary in this paper.  In the formulas below, $\defn{\varphi}$ is Euler's totient function and \defn{$J_2$} is Jordan's totient function $J_2(n) := \sum_{r \mid n} \mu(n/r) r^2$.
\begin{proposition}[{part of \cite[Thm.~5.2]{DLM1}}]
    \label{prop:formula for Ftr in the combinatorial case}
    Let $g\in G(m,p,n)$ and let $\lambda = \langle \lambda_1, \ldots, \lambda_k\rangle$ be the cycle type of the underlying permutation $\pi_{m/1}(g)$. 
    If $p = 1\neq m$, let $ a= \gcd(\col(w), m)$; then
    \[
    \Ftr_{m,1,n}(g) = \begin{cases}
    n(n+k-1)\cdot m^{k-1}  \cdot H_0(\lambda), & \text{ if } a =1,\\[6pt]
    \dfrac{n^2 (n+k)(n+k-1)m^{k} }{2} \cdot \dfrac{\varphi(a)}{a} \cdot H_0(\lambda), & \text{ otherwise.} 
    \end{cases}
    \]
    If $p = m$, suppose that the $k$ cycles of $g$ have colors $a_1, \ldots, a_k$, and let $d = \gcd(a_1, \ldots, a_k, m)$. Then 
    \[
\Ftr_{m,m,n}(g) = \begin{cases}
    m^{k-1} \cdot H_0(\lambda), & \text{ if } d = 1, \\[6pt]
    m^{k+1} \cdot \dfrac{J_2(d)}{d^2} \cdot  H_1(\lambda), & \text{ otherwise.}
    \end{cases}
    \]
\end{proposition}

\subsection{Coxeter elements and quasi-Coxeter elements}
\label{sec:cox and q-cox}

In a Coxeter group $W$, the product of the simple reflections in any order is called a \defn{Coxeter element} of $W$.  In particular, every Coxeter element has a reduced reflection factorization that generates $W$.  More generally, if $g$ is an element of a complex reflection group such that $\lR(g)$ is equal to the rank of $W$ and $g$ has a reduced reflection factorization that generates $W$, then $g$ is a \defn{quasi-Coxeter element} for $W$. (By definition, such elements only exist in well generated complex reflection groups.)  For (much) more on the motivation for this definition, see \cite[\S3]{DLM2}.

The enumeration of reduced reflection factorizations of a Coxeter element is given by the beautiful Arnold--Bessis--Chapoton formula   
\begin{equation}
\Fred(c)=\dfrac{h^n \cdot n!}{\#W},
\label{eq: deligne-arnold-bessis formula}
\end{equation}
where $h$ is the Coxeter number of $W$ (the multiplicative order of the Coxeter elements) and $n$ is its rank. (For a detailed discussion of the history of this formula, see \cite[\S1]{chapuy_theo_lapl}.)  In particular, in the symmetric group $\Symm_n$, whose rank is $n - 1$, whose Coxeter elements are the $n$-cycles, and whose Coxeter number is $n$, this recovers the number $H_0(n) = n^{n - 2}$ of reduced factorizations of an $n$-cycle (a special case of Theorem~\ref{thm:S_n genus 0}).

The next result, from \cite{DLM2}, gives the analogous enumeration of reduced reflection factorizations for quasi-Coxeter elements in the infinite families of well generated complex reflection groups. 
\begin{corollary}[{\cite[Cor.~6.11]{DLM2}}]
\label{cor: Fred for infinite families}
Let $g$ be a quasi-Coxeter element in the complex reflection group $W$.
\begin{enumerate}[(i)]

\item For $W=\Symm_n = G(1,1,n)$, we have that 
\[
\Fred(g) = \Ftr_W(g) = n^{n-2}.
\]
\item For $W=G(m,1,n)$ with $m > 1$, we have that 
\[
\Fred(g) = \Ftr_W(g) = n^n.
\]
\item For $W=G(m,m,n)$ with $m > 1$, if $g$ consists of two cycles of lengths $a$ and $b$ with colors that generate $\ZZ/m\ZZ$, then we have that
\[
\Fred(g) = \Ftr_W(g) = m(n-1)\cdot \binom{n-2}{a-1,b-1}\cdot a^a \cdot b^b.
\]
\end{enumerate}
\end{corollary}

The reduced reflection factorizations of quasi-Coxeter elements have a useful connectedness property that we describe next.  For any group $G$, a generator $\sigma_i$ of the $k$-strand braid group acts on $G^k$ via a \defn{Hurwitz move}:
\[
\sigma_i(g_1, \cdots, \quad g_i, \quad g_{i+1}, \quad \cdots,g_k)=(g_1,\cdots, \quad g_{i+1}, \quad g^{-1}_{i+1}g_ig_{i+1}, \quad \cdots,g_k),
\]
swapping two adjacent elements and conjugating one by the other. We call this the \defn{Hurwitz action} of the braid group on $G^k$.  Each Hurwitz move preserves the product $g_1 \cdots g_k$ of the tuple, the subgroup $
\langle g_1,\ldots,g_k\rangle$ that is generated by the elements $g_i$, and the multiset of conjugacy classes they determine. In particular, this means that for a reflection group $W$, there is a well defined Hurwitz action on the set of length-$k$ reflection factorizations of an element $g\in W$. The next result shows that this action is particularly nicely behaved in the infinite family.

\begin{proposition}[{\cite[Cor.~5.4]{LW}}]
\label{Prop: LW G(m, p, n) transitivity}
Let $W = G(m, p, n)$ and let $g\in W$ be an element with a reduced factorization that generates $W$. Then the Hurwitz action of the braid group on the set of reduced reflection factorizations of $g$ is transitive.
\end{proposition}

\subsection{Parabolic subgroups and parabolic quasi-Coxeter elements}
\label{sec:parabolic}

If $W$ is a complex reflection group acting on $V$ and $U \subseteq V$ is any subset, the pointwise stabilizer $W_U$ is called a \defn{parabolic subgroup} of $W$.\footnote{In the case of a real reflection group, this is different from the usual definition that a parabolic subgroup is a subgroup generated by a subset of simple reflections (as in, for example, \cite[\S2.4]{BjornerBrenti}); these subgroups are sometimes called \emph{standard parabolic subgroups}. The collection of parabolic subgroups (under our definition) agrees with the collection of all subgroups conjugate to a standard parabolic---see \cite[\S5-2]{Kane}.}  While it is not obvious from the definition, the next theorem shows that parabolic subgroups are reflection subgroups.

\begin{theorem}[{Steinberg's theorem \cite[Thm.~1.5]{Steinberg}}]\label{Thm: steinberg}
Let $W$ be a complex reflection group acting on $V$ and $U\subseteq V$ a subset. Then the pointwise stabilizer of $U$ is a reflection group, generated by those reflections in $W$ whose reflection hyperplanes contain $U$.
\end{theorem}

For any element $g\in W$, we define the \defn{parabolic closure} $W_g$ of $g$ to be the intersection of all parabolic subgroups that contain $g$. In this case the flat $X$ indexing $W_g$ is the fixed space $X=V^g$, and in particular we have $\rank(W_g)=\codim(V^g)$. We next record some facts from the literature concerning the behavior of parabolic subgroups that will be of use to us later.

\begin{proposition}[{\cite[Thm.~4.1]{taylor}}]
\label{Prop: Taylor_K=<H,t>}
Assume that $W$ is a complex reflection group and that $K$ and $H$ are reflection subgroups of $W$ such that $K=\langle H,t\rangle$ for some reflection $t\in W$. If $K$ is parabolic and its rank is greater than that of $H$, then $H$ is also parabolic.
\end{proposition}

\begin{lemma}
\label{Lem: g->gt}
Let $W$ be a real reflection group, $g$ any of its elements, and $W_g$ the parabolic closure of $g$. For a reflection $t\in W$, the following hold:
\begin{enumerate}
\item $t\leq_{\RRR} g$ if and only if $V^t\supset V^g$.
\item $t\leq_{\RRR} g$ if and only if $t\in W_g$.
\item $\lR(gt)=\lR(g)-1$ if and only if $t\in W_g$.
\item $\lR(gt)=\lR(g)+1$ if and only if $t\notin W_g$.
\end{enumerate} 
\end{lemma}
\begin{proof}
We will first show that all four statements are equivalent to each other. Indeed, parts (1) and (2) are equivalent after Steinberg's theorem (Theorem~\ref{Thm: steinberg}) and since $W_g=W_{V^g}$. By definition (see \S\ref{sec:reflen}), we have that $t\leq_{\RRR} g$ if and only if $\lR(gt)=\lR(g)-1$. Since moreover by parity considerations the only two possibilities for $\lR(gt)$ are the ones listed in (3) and (4), all four statements are equivalent to each other.
Now, the first statement is implicit in the proof of \cite[Lem.~2]{carter} (which is stated for Weyl groups but works verbatim for all real reflection groups \cite[Lem.~1.2.1]{bessis-dual-braid}).
\end{proof}

If $W$ is a complex reflection group, $H$ is a parabolic subgroup, and $g$ is a quasi-Coxeter element for $H$, then we say that $g$ is a \defn{parabolic quasi-Coxeter element} for $W$.  These elements have many attractive properties and were the main object of study in \cite{DLM2}.  In the case that $W$ is a Weyl group, the following theorem characterizes the parabolic quasi-Coxeter elements in $W$.

\begin{theorem}[{\cite[Thm.~3.4]{DLM2}}] 
\label{Prop: pqCox characterization for Weyl W}
Let $W \leq \GL(V)$ be a Weyl group of rank $n$, $g$ an element of $W$, and $W_g$ the parabolic closure of $g$. Then the following statements are equivalent. 
\begin{enumerate}[(i)]
    \item $g$ is a quasi-Coxeter element (respectively, parabolic quasi-Coxeter element).
    \item There exists a reduced reflection factorization $g=t_1\cdots t_k$ for which the associated roots $\rho_{t_i}$ and coroots $\widecheck{\rho}_{t_i}$ form $\ZZ$-bases of the root and coroot lattices of $W$ (resp., of $W_g$).
    \item $g$ satisfies $|\det(g - \III_V)|=I(W)$, where $\III_V$ is the identity on $V$ and $I(W)$ the connection index of $W$ (resp., $|\op{pdet}(g - \III_V)|=I(W_g)$ where $\op{pdet}$ denotes the pseudo-determinant).
    \item $g$ does not belong to any proper reflection subgroup of $W$ (resp., of $W_g$).
\end{enumerate}
\end{theorem}

Just as a permutation can be decomposed as a product of disjoint cycles, a parabolic quasi-Coxeter element $g$ with parabolic closure $W_g = W_1 \times \cdots \times W_r$ can be decomposed as a product $g = g_1 \cdots g_r$ where each $g_i$ is quasi-Coxeter for the irreducible factor $W_i$.  This \defn{generalized cycle decomposition} is the unique expression for $g$ as a reflection-length-additive product of commuting factors that cannot be further decomposed \cite[Prop.~3.13]{DLM2}.

Every parabolic subgroup of the symmetric group $\Symm_n$ is conjugate to a Young subgroup 
\[
\Symm_{\lambda}:=\Symm_{\lambda_1}\times\cdots\times\Symm_{\lambda_k},
\]
for some partition $\lambda$ of $n$, and the parabolic quasi-Coxeter elements associated with subgroups of type $\Symm_\lambda$ are the elements of cycle type $\lambda$.  The following theorem extends this description to cover parabolic quasi-Coxeter elements in the infinite family.  It is a combination of \cite[Thm.~3.11]{taylor} (which gives the parabolic subgroups), \cite[Cor.~3.16]{DLM2} (which gives the corresponding parabolic quasi-Coxeter elements), and \cite[Thm.~4.4]{Shi2007} (which gives a combinatorial formula for reflection length in $G(m, p, n)$).

\begin{theorem}
\label{thm:parabolic everything in G(m, p, n)}
For $m > 1$, every parabolic subgroup of $G(m, p, n)$ is either conjugate to a subgroup of the form 
\[
G(m, p, \lambda_0) \times \Symm_{\lambda_1} \times \cdots \times \Symm_{\lambda_k}
\]
for some partition\footnote{Note that $\lambda_0$ could be out of order in $\lambda$.} $(\lambda_0, \ldots, \lambda_k)$ of $n$ or conjugate by an element of $G(m, 1, n)$ to a subgroup of the form
\[ 
\Symm_{\lambda_1} \times \cdots \times \Symm_{\lambda_k}
\]
for some partition $(\lambda_1, \ldots, \lambda_k)$ of $n$.  If $g$ is an element of $G(m, 1, n)$, then
\begin{enumerate}[(i)]
\item $g$ is parabolic quasi-Coxeter for a subgroup of type $\Symm_{\lambda_1} \times \cdots \times \Symm_{\lambda_k}$ if and only if $g$ has cycles of lengths $\lambda_1, \ldots, \lambda_k$, all of color~$0$, and 
\item $g$ is parabolic quasi-Coxeter for a subgroup of type $G(m, 1, \lambda_0) \times \Symm_{\lambda_1} \times \cdots \times \Symm_{\lambda_k}$ if and only if $g$ has one cycle of length $\lambda_0$ whose color generates $\ZZ/m\ZZ$ and $k$ cycles of lengths $\lambda_1, \ldots, \lambda_k$ and color $0$. 
\end{enumerate}
If instead that $g$ is an element of $G(m, m, n)$, then 
\begin{enumerate}[(i)]
\addtocounter{enumi}{2}
\item $g$ is parabolic quasi-Coxeter for a subgroup of type $\Symm_{\lambda_1} \times \cdots \times \Symm_{\lambda_k}$ if and only if $g$ has cycles of lengths $\lambda_1, \ldots, \lambda_k$, all of color~$0$, and
\item $g$ is parabolic quasi-Coxeter for a subgroup of type $G(m, m, \lambda_0) \times \Symm_{\lambda_1} \times \cdots \times \Symm_{\lambda_k}$ if and only if $g$ has two cycles whose colors generate $\ZZ/m\ZZ$ and sum to $0$ and whose lengths add to $\lambda_0$, and has cycles of lengths $\lambda_1, \ldots, \lambda_k$ of color $0$.
\end{enumerate}
In all four cases, the reflection length of $g$ is $\lR(g) = n - k$.
In all cases except the last, the generalized cycles of $g$ are the cycles of $g$; in the case that $W_g \cong G(m, m, \lambda_0) \times \Symm_{\lambda_1} \times \cdots \times \Symm_{\lambda_k}$, the two cycles of nonzero color in $g$ together form a generalized cycle.
\end{theorem}

\begin{remark}[{essentially \cite[Rem.~6.15]{DLM2}}]
\label{rem:Fred extends}

By \cite[Prop.~2.11, Cor.~2.12, Prop.~3.3]{DLM2}, if $g$ is a parabolic quasi-Coxeter element in $W$ with parabolic closure $W_g$, then $\lR[W](g) = \lR[W_g](g)$ and $\Fred[W](g) = \Fred[W_g](g)$.  Therefore the formulas in Corollary~\ref{cor: Fred for infinite families} extend to give the numbers $\Fred[W](g)$ for any \emph{parabolic} quasi-Coxeter element in the infinite family.
    
\end{remark}

\subsection{Reflection generating sets}
\label{sec:rgs} 

In this section we include the results on reflection generating sets that we will use in the proofs of our main theorems. 

In \cite[\S1]{DLM2} we introduced the following combinatorial objects that appear in our uniform formulas. A \defn{good generating set} for a complex reflection group $W$ is a set of $\rank(W)$-many reflections that generate the full group $W$; the collection of good generating sets is denoted \defn{$\RGS(W)$}. Only well generated groups possess good generating sets. The next proposition shows that good generating sets are intimately related to quasi-Coxeter elements.

\begin{proposition}[{\cite[Cor.~5.5]{DLM2}}]
\label{Prop: prod_t_i is q-Cox}
Let $\{ t_1,\ldots, t_n\}$ be a good generating set for a well generated complex reflection group $W$ of rank $n$.  Then the product of the generators $ t_i$ in any order is a quasi-Coxeter element for $W$.
\end{proposition}

We also consider a generalization of the notion of good generating sets.  We say that a set of $\big(\rank(W)-\lR(g)\big)$-many reflections that generate the full group $W$ \emph{when combined with a reduced reflection factorization of $g$} is a \defn{relative generating set} for the element $g \in W$, and we denote by \defn{$\RGS(W, g)$} the collection of relative generating sets for $g$.

One of the main theorems of \cite{DLM2} was the following characterization of parabolic quasi-Coxeter elements in terms of relative generating sets, absolute order, and full reflection length.

\begin{theorem}[{\cite[Thm.~5.10]{DLM2}}]
\label{Prop: characterization of pqCox}
For a well generated complex reflection group $W$ of rank $n$ and an element $g\in W$, the following are equivalent.
\begin{enumerate}[(i)]
    \item The element $g$ is a parabolic quasi-Coxeter element of $W$.
    \item The collection $\RGS(W,g)$ of relative generating sets with respect to $g$ is nonempty.
    \item There exists a quasi-Coxeter element $w\in W$ such that $g\leq_{\RRR} w$.
    \item The full reflection length of $g$ satisfies $\ltr(g)=2n-\lR(g)$.
\end{enumerate}
\end{theorem}

We now describe the good generating sets in $G(m, 1, n)$ and $G(m, m, n)$.

\subsubsection*{The combinatorial family} 
It is natural to represent collections of reflections in $G(m, p, n)$ as graphs on the vertex set $\{1, \ldots, n\}$: a transposition-like reflection with underlying permutation $(i j)$ is represented by an edge joining $i$ to $j$, while a diagonal reflection whose matrix entry of nonzero color occurs at position $(i, i)$ is represented by a loop at vertex $i$.  
We say that a collection of reflections in $G(m, p, n)$ is \defn{connected} if the associated graph is connected. 

A connected graph on $n$ vertices with $n$ edges contains a unique cycle.  In the case that the cycle is a loop, we call the graph is a \defn{rooted tree}.  Otherwise, the cycle contains at least two vertices, and in this case we say that the graph is a \defn{unicycle}.  Given a set $S$ of $n$ reflections whose associated graph is a unicycle, define a statistic $\delta(S)$, as follows: let $i_0, i_1, \ldots, i_{k - 1}, i_k = i_0$ be the vertices of the unique (graph) cycle, oriented in one of the two possible cyclic orders arbitrarily, so that the associated reflections are $[(i_j \; i_{j + 1}); a_j]$ for some $a_0, \ldots, a_{k - 1} \in \ZZ/m\ZZ$; then set
\[
\delta(S) = a_0 + \ldots + a_{k - 1}.
\]
Changing the cyclic orientation of the cycle replaces $\delta(S)$ with $-\delta(S)$.  This distinction will never be important, as we explain in the following remark.

\begin{remark}
\label{rem: Shi's delta}
It is not difficult to see (for example, following the argument in \cite[Prop.~3.30]{LW}, conjugating by a diagonal element of $G(m, 1, n)$ to send all but one of the elements of $S$ to true transpositions) that if $S$ is a set of reflections whose associated graph is a unicycle, then multiplying the factors of $S$ together in any order produces an element of $G(m, m, n)$ with exactly two (permutation) cycles, one of color $\delta(S)$ and the other of color $-\delta(S)$. This is compatible with the arbitrary choice in the definition of $\delta$.
\end{remark}

In \cite{DLM2} we described and enumerated the relative generating sets of parabolic quasi-Coxeter elements in $G(m, 1, n)$ and $G(m, m, n)$, building on earlier work of Shi \cite{Shi2005}.  In order to state these results, we introduce some additional terminology.

Given a parabolic quasi-Coxeter element $g$ of $G(m, 1, n)$ or $G(m, m, n)$, the \defn{partition $\Pi_g$ induced by $g$} is the partition of $\{1, \ldots, n\}$ whose blocks are the support of the generalized cycles of $g$.

Given a partition $\Pi$ of the set $K$, we say that a graph $\Gamma$ with vertex set $K$ is a \defn{tree relative to $\Pi$} if no edge of $\Gamma$ connects two vertices in the same block of $\Pi$ and, furthermore, contracting each block of $\Pi$ to a single point leaves a tree.  We say that $\Gamma$ is a \defn{rooted tree relative to $\Pi$} if $\Gamma$ is the union of a tree with respect to $\Pi$ and a single loop (at any vertex of $K$).  Finally, we say that $\Gamma$ is a \defn{unicycle relative to $\Pi$} if $\Gamma$ is the union of a tree with respect to $\Pi$ and a single non-loop edge.

\begin{proposition}[{\cite[Prop.~4.14]{DLM2}}]
\label{prop:RGS characterization}
Suppose that $W$ is either $G(m, 1, n)$ or $G(m, m, n)$ and that $g$ is a parabolic quasi-Coxeter element for $W$.  Let $\Pi_g$ be the partition of $\{1, \ldots, n\}$ induced by $g$.
\begin{enumerate}[(i)]
\item If either $W = G(m, 1, n)$ and $g$ is a quasi-Coxeter element for the subgroup $G(m, 1, \lambda_0) \times \Symm_{\lambda_1} \times \cdots \times \Symm_{\lambda_k}$, or $W = G(m, m, n)$ and $g$ is a quasi-Coxeter element for the subgroup $G(m, m, \lambda_0) \times \Symm_{\lambda_1} \times \cdots \times \Symm_{\lambda_k}$, then a set $S$  of reflections is a relative generating set for $g$ if and only if the graph associated to $S$ is a tree relative to $\Pi_g$.
\item If $W = G(m, 1, n)$ and $g$ is a quasi-Coxeter element for the subgroup $\Symm_{\lambda_1} \times \cdots \times \Symm_{\lambda_k}$, then a set $S$ of reflections is a relative generating set for $g$ if and only if the graph associated to $S$ is a rooted tree relative to $\Pi_g$ and the color of the unique diagonal reflection in $S$ generates $\ZZ/m\ZZ$.
\item  If $W = G(m, m, n)$ and $g$ is a quasi-Coxeter element for the subgroup $\Symm_{\lambda_1} \times \cdots \times \Symm_{\lambda_k}$, then a set of reflections is a relative generating set for $g$ if and only if the graph associated to $S$ is a unicycle relative to $\Pi_g$ and a certain color $c$ is a primitive generator for $\ZZ/m\ZZ$, where $c$ is defined as follows: if contracting the cycles of $g$ leaves a rooted tree, then $c$ is the color of the loop edge, whereas if contracting the cycles of $g$ leaves a unicycle, then $c$ is the value of the statistic $\delta$ on the cycle in the contracted graph.
\end{enumerate}
\end{proposition}

\begin{theorem}[{\cite[Thm.~7.3]{DLM2}}]
\label{thm:count rgs}
Suppose that $m > 1$, $W$ is either $G(m, 1, n)$ or $G(m, m, n)$, and $g$ is a parabolic quasi-Coxeter element for $W$.
\begin{enumerate}[(i)]
\item If either $W = G(m, 1, n)$ and $g$ is a quasi-Coxeter element for the subgroup $G(m, 1, \lambda_0) \times \Symm_{\lambda_1} \times \cdots \times \Symm_{\lambda_k}$, or $W = G(m, m, n)$ and $g$ is a quasi-Coxeter element for the subgroup $G(m, m, \lambda_0) \times \Symm_{\lambda_1} \times \cdots \times \Symm_{\lambda_k}$, then 
\begin{equation} \label{eq: RGS first family}
\# \RGS(W, g) = m^k \cdot n^{k - 1} \cdot \prod_{i = 0}^k \lambda_i.
\end{equation}
\item If $W = G(m, 1, n)$ and $g$ is a quasi-Coxeter element for the subgroup $\Symm_{\lambda_1} \times \cdots \times \Symm_{\lambda_k}$, then 
\begin{equation}
\# \RGS(W, g) = \varphi(m) \cdot m^{k - 1} \cdot n^{k - 1} \cdot \prod_{i = 1}^k \lambda_i,
\end{equation}
where $\varphi$ denotes Euler's totient function. 
\item  If $W = G(m, m, n)$ and $g$ is a quasi-Coxeter element for the subgroup $\Symm_{\lambda_1} \times \cdots \times \Symm_{\lambda_k}$, then 
\begin{equation}
\#\RGS(W, g) =  \frac{\varphi(m) \cdot m^{k - 1}}{2}  \cdot
\Big(n^k - n^{k-1} -\sum_{j=2}^k  (j-2)! \cdot n^{k-j} e_j(\lambda)\Big) \cdot  \prod_{i = 1}^k \lambda_i,
\end{equation}
where $e_i(\lambda)$ denotes the $i$th elementary symmetric function in the variables $\lambda=(\lambda_1,\ldots,\lambda_k)$.
\end{enumerate}
\end{theorem}

\begin{remark} \label{rem: RGS types ABD}
In particular, we have that $\#\RGS(\Symm_n)=n^{n-2}$ and  $\#\RGS(B_n)=(2n)^{n-1}$ \cite[\href{https://oeis.org/A052746}{A052746}]{oeis}. The sequence $(\#\RGS(D_n))$ is  \cite[\href{https://oeis.org/A320064}{A320064}]{oeis} (see also \cite[Prop.~5.1]{ardila2020}).
\end{remark}

\section{Proof of the main results}
\label{Sec: proof main}

In this section, we prove our main results.  In the case of Weyl groups, the statement is as follows.

\weyltheorem

Theorem~\ref{Thm: Weyl case} has an even better expression for parabolic \emph{Coxeter} elements, where the product structure and analogy to Theorem~\ref{thm:S_n genus 0} is particularly apparent.  It follows immediately via \eqref{eq: deligne-arnold-bessis formula}.

\weylthmdecompirred

We generalize Theorem~\ref{Thm: Weyl case} to well generated complex reflection groups below, as Theorem~\ref{volume theorem}.  In order to state it, we need to introduce some additional terminology.

The reader is invited to recall at this point the construction of roots and coroots associated to complex reflection groups given in Section~\ref{sec:Root systems}. In its most general form, our main Theorem~\ref{volume theorem} involves a Grammian statistic on roots associated to (relative) generating sets, which is defined as follows.

\begin{definition}
\label{Defn: Gram Det}
The \defn{Grammian determinant $\GD(\bm\rho_{\ttt})$}  of a set of roots $\bm{\rho_{ t}}:=\{\rho_{ t_i}\}$ associated to a set of reflections $\ttt:=\{ t_i\}\subset \RRR$ is defined as the determinant of the Gram matrix $\big(\langle\rho_{ t_i},\widecheck{\rho}_{ t_j}\rangle\big)_{i,j}$, i.e., 
\[
\GD(\bm{\rho_{\ttt}}):=\det\big( \langle \rho_{ t_i},\widecheck{\rho}_{ t_j}\rangle\big)_{i,j}
\]
where $\rho_{t_i}$ and $\widecheck{\rho}_{t_j}$ are the root and coroot associated with the reflections $t_i$ and $t_j$, respectively. 
\end{definition}

With this terminology in hand, we restate our main theorem for well generated complex reflection groups.

\volumetheorem

\begin{example}
\label{Ex: H_3 calc}
We illustrate here Theorem~\ref{thm:main} for the identity element $\id$ in the group $W=H_3$. There are $380$ good generating sets $\ttt$ for $H_3$ and they are naturally divided into three classes with respect to their Grammian statistics $\GD(\bm \rho_{\ttt})$. One of these classes has $180$ elements, each with Grammian statistic equal to $2$, while the other two classes have $100$ elements each\footnote{These two classes are related via a reflection automorphism, as in \cite{RRS}, which is why they have the same size.  That the total contribution to the sum from the two classes is rational (even though each class contributes an irrational number) follows from the construction of reflection automorphisms from Galois automorphisms of the field of definition of the group.} and statistics equal to $\GD(\bm \rho_{\ttt})=3+\sqrt{5}$ and $\GD(\bm \rho_{\ttt})=3-\sqrt{5}$. Putting it all together, Theorem~\ref{thm:main} states that 
\begin{align*}
\Ftr_{H_3}(\id)&=6!\cdot\Bigg(100\cdot\left(\dfrac{1}{3+\sqrt{5}}+\dfrac{1}{3-\sqrt{5}}\right)+180\cdot\dfrac{1}{2}\Bigg)\\
&=6!\cdot \left(100\cdot\dfrac{3}{2}+180\cdot\dfrac{1}{2}\right)=6!\cdot 240=172800,
\end{align*}
which agrees with the representation-theoretic calculation of $\Ftr_{H_3}(\id)$ in \cite[Rem.~5.2]{DLM1}.
\end{example}

\begin{remark}
\label{Rem: fred(g_i)->fred(g)}
 
By \cite[Cor.~6.14]{DLM2}, we have that
\[
\Fred(g)=\binom{\lR(g)}{\lR(g_1),\ldots,\lR(g_r)}\cdot\prod_{i=1}^r\Fred(g_i),
\]
and so \eqref{Eq: weyl thm} and \eqref{eq:volume theorem} may be equivalently written as
\begin{align}
\label{altform}
\Ftr_W(g) & =\frac{\ltr(g)!}{\lR(g)!}\cdot\Fred(g)\cdot 
\#\RGS(W,g)\cdot\dfrac{I(W_g)}{I(W)} \notag\\ \intertext{for a Weyl group $W$, and}
\Ftr_W(g) &= \frac{\ltr(g)!}{\lR(g)!} \cdot \Fred(g) \cdot \sum_{ \ttt \in \RGS(W, g)} \frac{ \GD(\bm{\rho}_g)}{\GD(\bm{\rho_{\ttt}}\cup\bm{\rho}_g)}
\end{align}
for any complex reflection group $W$.
\end{remark}

\subsection*{Outline of the proof}

The proof strategy of Theorems~\ref{Thm: Weyl case} and \ref{thm:main} is as follows. First, we show that Theorems~\ref{Thm: Weyl case} and~\ref{volume theorem} are equivalent when $W$ is a Weyl group. Next, we do the following reductions:
\begin{itemize}
    \item In Section~\ref{sec:reduce to irreducible} we show that it suffices to look at the case of irreducible groups $W$. 
    \item A priori, for a given group $W$, the value of the right side of \eqref{eq:volume theorem} depends on a variety of choices, including the choice of root system for $W$ and the choice of reduced reflection factorization of $g$.  In Section~\ref{sec:invariance} we show that these choices do not affect the value of the right side of \eqref{eq:volume theorem}, and also establish two other related invariance properties.
\end{itemize}
 After these reductions, we then proceed to prove Theorem~\ref{thm:main} in a case-by-case fashion.  The case of the symmetric group $\Symm_n = G(1,1,n)$ was verified in Section~\ref{intro: recovering hurwitz}.  In Sections~\ref{sec: proof Gm1n} and \ref{sec: proof Gmmn}, we respectively prove the theorem in the infinite families $G(m, 1, n)$ and $G(m, m, n)$ with $m>1$.  In Section~\ref{Sec: cut-and-join counting}, we give a recurrence relation (Theorem~\ref{thm:cut and join}) inspired by the cut-and-join equations to enumerate full reflection factorizations in real reflection groups.  Finally, in Section~\ref{sec: proof exceptional}, we complete the proof in the exceptional groups using a combination of our cut-and-join recurrence and a large computer calculation.

\subsection{The main theorems are equivalent for Weyl groups}
\label{sec: equivalence of main thms for Weyl groups}

To establish the equivalence of our two main theorems for Weyl groups, we make use of a characterization of good generating sets due to Baumeister and Wegener.

\begin{proposition}[{\cite[Cor.~1.2]{BW}}]
\label{Prop: BW charn of gen sets}
For a Weyl group $W$ of rank $n$, a set  $\{t_i\}_{i=1}^n$ of $n$ reflections generates $W$ if and only if the roots $\{\rho_{t_i}\}$ and coroots $\{\widecheck{\rho}_{t_i}\}$ respectively form $\ZZ$-bases of the root and coroot lattices $\mathcal{Q}$ and $\widecheck{\mathcal{Q}}$ of $W$.
\end{proposition}

\begin{proof}[Proof of equivalence of Theorems~\ref{Thm: Weyl case} and~\ref{thm:main} for Weyl groups]
It is sufficient to show that for any Weyl group $W$, any parabolic quasi-Coxeter element $g\in W$, any choice $\bm\rho_g$ of roots associated to a reduced factorization of $g$, and any relative generating set $\ttt\in \RGS(W,g)$, we have
\[
\frac{\GD(\bm{\rho}_g)}{\GD(\bm\rho_{\ttt}\cup\bm\rho_g)} = \frac{I(W_g)}{I(W)}.
\]
By Proposition~\ref{Prop: BW charn of gen sets}, the sets $\bm\rho := \bm\rho_{\ttt}\cup\bm\rho_g$ and $\widecheck{\bm\rho} := \widecheck{\bm\rho}_{\ttt}\cup\widecheck{\bm\rho}_g$ form $\ZZ$-bases of the root and coroot lattices $\mathcal{Q}$ and $\widecheck{\mathcal{Q}}$ of $W$, and similarly $\bm\rho_g$ and $\widecheck{\bm\rho}_g$ are $\ZZ$-bases for the root and coroot lattices for $W_g$.  
Let $\bm\alpha:=\{\alpha_i\}$ be a set of simple roots for $W$, with associated coroots $\widecheck{\bm\alpha}:=\{\widecheck{\alpha}_i\}$.  Since $\bm\rho$ and $\bm\alpha$ are $\ZZ$-bases of the same lattice, the change-of-basis matrix $U$ between them has determinant $\pm 1$, and likewise for the change-of-basis matrix $\widecheck{U}$ between $\widecheck{\bm\rho}$ and $\widecheck{\bm\alpha}$.  Moreover, since each $\widecheck{\rho}_i$ is a positive multiple of the corresponding $\rho_i$, and likewise for the $\alpha_i$, we have $\det(U) = \det(\widecheck{U})$.  Thus
\[
\GD(\bm\rho) = 
\det\big(\langle\rho_i,\widecheck{\rho}_j\rangle\big)= 
\det(U) \cdot \det\big(\langle \alpha_i, \widecheck{\rho}_j\rangle\big) = 
\det(U) \cdot \det\big(\langle \alpha_i, \widecheck{\alpha}_j\rangle\big) \cdot \det\big(\widecheck{U}^\top\big) = \GD(\bm\alpha).
\]
It is a standard fact (as in \cite[\S9-4]{Kane}, for instance) that $\det\big(\langle\alpha_i,\widecheck{\alpha}_j\rangle\big)$ is equal to the connection index $I(W)$, so that $\GD(\bm\rho_{\ttt}\cup\bm\rho_g)=I(W)$.  For the same reasons, $\GD(\bm\rho_g)=I(W_g)$.  The desired equivalence follows immediately.
\end{proof}

\subsection{Reduction to the irreducible case}
\label{sec:reduce to irreducible}

Suppose that $g$ is a parabolic quasi-Coxeter element in a reducible complex reflection group $W = W_1 \times W_2$, with $g = g_1g_2$ the decomposition of $g$ in the direct product.  (This is not necessarily the generalized cycle decomposition of $g$, which might further refine $g_1$ and $g_2$.)  Since the set of $W$-reflections decomposes as the union of the sets $\RRR_1$ of $W_1$-reflections and $\RRR_2$ of $W_2$-reflections, we have that a $\RRR$-factorization $g = t_1 \cdots t_k$ is reduced if and only if it is formed by shuffling together a reduced $\RRR_1$-factorization of $g_1$ and a reduced $\RRR_2$-factorization of $g_2$.  Thus
\[
\Fred(g) = \frac{\lR(g)!}{\lR[W_1](g_1)! \cdot \lR[W_2](g_2)!} \Fred[W_1](g_1) \Fred[W_2](g_2)
.
\]
The same is true if one replaces the word ``reduced'' with the phrase ``minimum-length full'', and so one may replace each copy of ``red'' with ``full'' in the preceding equation.

Similarly, it is easy to see that every relative generating set $\ttt$ for $g$ is the union $\ttt = \ttt_1 \cup \ttt_2$ of an element $\ttt_1$ of $\RGS(W_1, g_1)$ with an element $\ttt_2$ of $\RGS(W_2, g_2)$.  Since reflections in $W_1$ commute with those in $W_2$, the Gram matrices associated to the numerator and denominator of any summand in \eqref{eq:volume theorem} are block-diagonal, and so
\[
\frac{ \GD(\bm{\rho}_g)}{\GD(\bm{\rho_{\ttt}}\cup\bm{\rho}_g)}
=
\frac{ \GD(\bm{\rho}_{g_1})}{\GD(\bm{\rho}_{\ttt_1}\cup\bm{\rho}_{g_1})}
\cdot 
\frac{ \GD(\bm{\rho}_{g_2})}{\GD(\bm{\rho}_{\ttt_2}\cup\bm{\rho}_{g_2})}.
\]
Summing over all relative generating sets for $g$, we have further that 
\[
\sum_{ \ttt \in \RGS(W, g)} 
\frac{ \GD(\bm{\rho}_g)}{\GD(\bm{\rho_{\ttt}}\cup\bm{\rho}_g)}
=
\Biggl(\sum_{ \ttt_1 \in \RGS(W_1, g_1)} 
\frac{ \GD(\bm{\rho}_{g_1})}{\GD(\bm{\rho}_{\ttt_1}\cup\bm{\rho}_{g_1})}\Biggr)
\cdot 
\Biggl(\sum_{ \ttt_2 \in \RGS(W_2, g_2)} 
\frac{ \GD(\bm{\rho}_{g_2})}{\GD(\bm{\rho}_{\ttt_2}\cup\bm{\rho}_{g_2})}\Biggr).
\]

Combining the calculations above, assuming that the main theorem holds for $W_1$ and $W_2$ (and employing it specifically in the form \eqref{altform}), we have
\begin{align*}
\Ftr_W(g) &= \frac{\ltr(g)!}{\ltr[W_1](g_1)! \cdot \ltr[W_2](g_2)!} \Ftr_{W_1}(g_1) \Ftr_{W_2}(g_2) \\
& = \frac{\ltr(g)!}{\ltr[W_1](g_1)! \cdot \ltr[W_2](g_2)!} \cdot \Biggl(\frac{\ltr(g_1)!}{\lR(g_1)!} \cdot \Fred(g_1) \cdot \sum_{ \ttt_1 \in \RGS(W, g_1)} \frac{ \GD(\bm{\rho}_{g_1})}{\GD(\bm{\rho}_{\ttt_1}\cup\bm{\rho}_{g_1})} \Biggr) \times  \\
& \qquad \times \Biggl(
\frac{\ltr(g_2)!}{\lR(g_2)!} \cdot \Fred(g_2) \cdot \sum_{ \ttt_2 \in \RGS(W, g_2)} \frac{ \GD(\bm{\rho}_{g_2})}{\GD(\bm{\rho}_{\ttt_2}\cup\bm{\rho}_{g_2})}\Biggr)\\
& = \frac{\ltr(g)!}{\lR(g)!} \cdot \Fred(g) \cdot \sum_{ \ttt \in \RGS(W, g)} \frac{ \GD(\bm{\rho}_g)}{\GD(\bm{\rho_{\ttt}}\cup\bm{\rho}_g)},
\end{align*}
as needed.  Thus, if the theorem is valid for \emph{irreducible} groups, it follows by induction that it holds for all complex reflection groups.  For this reason, in the remainder of the proof (beginning in Section~\ref{sec: proof Gm1n}), we consider only the irreducible groups, on a case-by-case basis.

\subsection{Invariance}
\label{sec:invariance}

In this section, we prove four invariances: first, invariance of $\GD(\bm{\rho}_{\ttt})$ with respect to the choice of root system for $W$; second, invariance of the right side of \eqref{eq:volume theorem} under conjugation; third, invariance of the quotient $\frac{ \GD(\bm{\rho}_g)}{\GD(\bm{\rho_{\ttt}}\cup\bm{\rho}_g)}$ with respect to the choice of reduced factorization for $g$; and finally, invariance of $\GD(\bm{\rho}_{\ttt})$ under conjugation of one reflection in $\ttt$ by another.

Fix a complex reflection group $W$ and an associated root system $\bm{\rho}_{\RRR}$.  First, observe that we may replace a root $\rho_{ t}$ with its multiple $c \rho_{ t}$ so long as we also replace the coroot $\widecheck{\rho}_{ t}$ with $  \widecheck{\rho}_{ t} / \overline{c}$.  It follows immediately that choosing a different root system $\widetilde{\bm{\rho}_{\RRR}}$ corresponds to multiplying any Gram matrix $\big(\langle\rho_{ t_i},\widecheck{\rho}_{ t_j}\rangle\big)_{i,j}$ on the left and right by inverse diagonal matrices, so that its determinant is unaffected.  

Second, to facilitate computations in subsequent sections, it will be convenient to consider elements $g$ having simple forms.  If $u$ is any unitary transformation that normalizes $W$ and $\bm{\rho}_{\RRR}$ is a root system for $W$, then conjugating any reflection factorization $\ttt$ of $g$ on the left by $u$ (factor by factor) produces a reflection factorization $u\ttt u^{-1}$ of $ugu^{-1}$, and likewise for a relative generating set $S$.  The transformed root system $u\bm{\rho}_{\RRR}$ is not necessarily equal to $\bm{\rho}_{\RRR}$, but since $u$ is unitary it is \emph{a} root system for $W$.  Since $u$ is unitary, the Gram matrix of $\ttt$ with respect to $\bm{\rho}_{\RRR}$ is equal to the Gram matrix of $u \ttt u^{-1}$ with respect to $u\bm{\rho}_{\RRR}$, which (by the previous paragraph) is equal to the Gram matrix of $u \ttt u^{-1}$ with respect to $\bm{\rho}_{\RRR}$.  Therefore, we may replace $g$ at will with a conjugate $ugu^{-1}$ without changing the value of the right side of \eqref{eq:volume theorem}.

Third, we explain why the ratio of Grammian determinants in \eqref{eq:volume theorem} is independent of the choice of reduced factorization for $g$.

\begin{lemma}
\label{Lem: GD ratio independence}
Keeping the notation of Theorem~\ref{volume theorem} and assuming that $\ttt$ is a fixed element of $\RGS(W,g)$, if $\bm{\rho}_g$ and $\bm{\rho'}_g$ are sets of roots associated with two different reduced reflection factorizations of $g$, then 
\[
\frac{ \GD(\bm{\rho}_g)}{\GD(\bm{\rho_{\ttt}}\cup\bm{\rho}_g)}=\frac{ \GD(\bm{\rho'}_g)}{\GD(\bm{\rho_{\ttt}}\cup\bm{\rho'}_g)}.
\]
\end{lemma}
\begin{proof}
The proof of this lemma relies on a standard interpretation of the ratio of Gram determinants. In the standard setting, however, we pair vectors with themselves (and not a rescaling as we do with the coroots in Definition~\ref{Defn: Gram Det}), so that we first have to make this transition clear.

The \emph{usual} Gram determinant $\widetilde{\GD}(\bm{v})$ of a set of vectors $\bm{v}:=\{v_1,\ldots,v_n\}$ is the determinant of the Gram matrix of $\bm{v}$; that is, $\widetilde{\GD}(\bm{v}):=\det\big(\langle v_i,v_j\rangle\big)_{i,j}$ for the Hermitian inner product $\langle \cdot , \cdot \rangle$. The two Gram matrices (the usual one and that of Definition~\ref{Defn: Gram Det}) are identical after a rescaling of their columns. For their determinants, this translates to 
\[
\GD(\bm{\rho})=\widetilde{\GD}(\bm{\rho})\cdot\prod_{i=1}^n\dfrac{1-\xi_i}{\langle \rho_i,\rho_i\rangle},
\]
where $\bm{\rho}:=\{\rho_1,\ldots,\rho_n\}$ is a set of roots and $\xi_i$ is the eigenvalue associated to the (reflection of the) root $\rho_i$. After rewriting the ratios of the statement in terms of the usual Gram determinant $\widetilde{\GD}$, the extra factors cancel out and we are left to show that 
\begin{equation}\label{eq:usual GD}
\frac{ \widetilde{\GD}(\bm{\rho}_g)}{\widetilde{\GD}(\bm{\rho_{\ttt}}\cup\bm{\rho}_g)}=\frac{ \widetilde{\GD}(\bm{\rho'}_g)}{\widetilde{\GD}(\bm{\rho_{\ttt}}\cup\bm{\rho'}_g)}.
\end{equation}
Now the corresponding reflections for $\bm{\rho}_g$ and $\bm{\rho'}_g$ form good generating sets for $W_g$, which means that the two sets of roots both span the orthogonal complement $X^{\perp}$ of the flat $X:=V^g$. Then \eqref{eq:usual GD} is an easy corollary of the standard formula (see for instance \cite[Thm.~8.7.4]{davis_GD})
\[
\dfrac{\widetilde{\GD}(x_1,x_2,\ldots,x_n,y)}{\widetilde{\GD}(x_1,x_2,\ldots,x_n)}=\min_{a_i\in\CC} \|y-(a_1x_1+\cdots + a_nx_n)\|,
\]
where $\|v\|:=\langle v,v\rangle$ and the $x_i$ are linearly independent vectors.
\end{proof}

Finally, we show that conjugating one element of $\ttt$ by another does not affect $\GD(\bm{\rho}_{\ttt})$.

\begin{proposition}
\label{prop: GD invariant to Hurwitz moves}
Let $W$ be a complex reflection group, let $\ttt = \{ t_1, \ldots,  t_k\}$ be any set of reflections in $W$, and let $\ttt' = (\ttt \smallsetminus \{ t_i\}) \cup \{ t_j  t_i  t_j^{-1}\}$ be the result of replacing $ t_i$ in $\ttt$ with its conjugate by $ t_j$ for any $i, j \in \{1, \ldots, k\}$.  Then $\GD(\bm{\rho}_{\ttt}) = \GD(\bm{\rho}_{\ttt'})$.
\end{proposition}
\begin{proof}
Since $ t_j$ is unitary, $ t_j(\rho_i)$ is orthogonal to the fixed plane of $ t_j t_i t_j^{-1}$, $t_j(\widecheck{\rho}_i)$ is a scalar multiple of $ t_j(\rho_i)$, and $\langle t_j(\rho_i), t_j(\widecheck{\rho}_i)\rangle = \langle \rho_i, \widecheck{\rho}_i\rangle$.  We have already seen above that the value of the Gram determinant does not depend on which root system we use, so we may as well assume that $\rho_{t_j t_i t_j^{-1}} = t_j(\rho_i)$, in which case it follows automatically that $\widecheck{\rho}_{t_j t_i t_j^{-1}} = t_j(\widecheck{\rho}_i)$.  
Since $ t_j$ is a reflection, $ t_j(\rho_i) = \rho_i - \langle \rho_i, \widecheck{\rho}_{ j}\rangle \rho_j$ and $ t_j(\widecheck{\rho}_i) = \widecheck{\rho}_i - \langle \widecheck{\rho}_i, \widecheck{\rho}_{ j}\rangle \rho_j$.  In terms of the Gram matrix, this means that the Gram matrix of $\ttt'$ can be formed from the Gram matrix of $\ttt$ by simultaneously adding a scalar (specifically, $\langle \rho_i, \widecheck{\rho}_{ j}\rangle$) times row $j$ to row $i$ and adding a scalar (specifically, $\langle \widecheck{\rho}_i, \widecheck{\rho}_{ j}\rangle$) times column $j$ to column $i$.  These elementary operations do not change the value of the determinant.
\end{proof}

We now move on the case-by-case proof of Theorem~\ref{thm:main}.

\subsection{The combinatorial family \texorpdfstring{$G(m, 1, n)$}{G(m, 1, n)} with \texorpdfstring{$m > 1$}{m > 1}}
\label{sec: proof Gm1n}

In this section, we prove Theorem~\ref{volume theorem} for the group $W := G(m, 1, n)$.
Let $g$ be a parabolic quasi-Coxeter element in $W$.  According to Theorem~\ref{thm:parabolic everything in G(m, p, n)}, there are two cases: the parabolic subgroup to which $g$ is associated is either of the form $G(m, 1, \lambda_0) \times \Symm_{\lambda_1} \times \cdots \times \Symm_{\lambda_k}$ or is a Young subgroup $\Symm_{\lambda_1} \times \cdots \times \Symm_{\lambda_k}$ for the symmetric group $\Symm_n \subset W$.

\medskip
\noindent  {\bf Case 1(a):} Suppose $g$ is a quasi-Coxeter element for the parabolic subgroup $G(m, 1, \lambda_0) \times \Symm_{\lambda_1} \times \cdots \times \Symm_{\lambda_k}$.
By Theorem~\ref{thm:parabolic everything in G(m, p, n)}(ii), $g$ has $k + 1$ cycles: a $\lambda_0$-cycle whose color generates $\ZZ/m\ZZ$ and $k$ other cycles of color $0$.  Thus $\gcd(\col(g), m) = 1$ and we have by Proposition~\ref{prop:formula for Ftr in the combinatorial case} that 
\begin{equation}   \label{eq:case 1(a) LHS}
\Ftr_W(g) = n(n+(k+1)-1)\cdot m^{(k+1)-1}  \cdot H_0(\lambda_0, \ldots, \lambda_k) = (n + k)! \cdot m^k \cdot n^{k - 1} \cdot \prod_{i = 0}^k \frac{\lambda_i^{\lambda_i}}{(\lambda_i - 1)!},
\end{equation}
where for the second equality we use the formula for Hurwitz numbers in Theorem~\ref{thm:S_n genus 0}.  It remains to compute the right side of \eqref{eq:volume theorem}.  We begin by considering the sum over $\RGS(W, g)$, taking advantage of the freedom of choice allowed by Section~\ref{sec:invariance}. 

The conjugacy class of $g$ contains the element 
\[
[(1 \cdots \lambda_0) ((\lambda_0 + 1) \cdots (\lambda_0 + \lambda_1)) \cdots ((n - \lambda_k + 1) \cdots n); (0, \ldots, 0, c, 0, \ldots, 0)],
\]
where the unique nonzero color is in position $\lambda_0$.  Since the sum being computed is preserved by conjugacy, we may as well replace $g$ with this element.  We choose the following fixed reduced factorization of $g$, consisting of one diagonal reflection and $n - k - 1$ \defn{adjacent transpositions} (i.e., reflections of the form $(i \ i+1) = [(i \ i+1); 0]$ for some $i$):
\begin{multline*}
g = [\id; (c, 0, \ldots, 0)] \cdot 
(12) \cdots (\lambda_0 - 1 \; \lambda_0) \quad \cdot \quad 
(\lambda_0 + 1 \; \lambda_0 + 2) \cdots ((\lambda_0 + \lambda_1 - 1)(\lambda_0 + \lambda_1)) \quad \cdots \\
\cdots \quad ((n - \lambda_k + 1)(n - \lambda_k + 2)) \cdots (n - 1 \; n).
\end{multline*}
We choose roots as follows: for the diagonal reflection, we take $\rho_0 = (1 - \xi, 0, \ldots, 0)$ and $\widecheck{\rho}_{0} = (1, 0, \ldots, 0)$, where $\xi = \zeta_m^c$ is the non-$1$ matrix entry of $g$.  For the adjacent transpositions, we take $\rho_i = (0, \ldots, 0, 1, -1, 0, \ldots, 0) = \widecheck{\rho}_i$.  The associated Gram matrix is block-diagonal: omitting $0$ entries, the first ($\lambda_0 \times \lambda_0$) block is
\begin{equation}
\label{eq:block}
A_{\lambda_0}=\begin{bmatrix} 
1 - \xi & 1 - \xi &  &  &  &  \\
1 & 2 & -1 &  &  &  \\
 & -1 & 2 & -1 & &  \\
 &  & -1 & 2  & \ddots &  \\
 &  & & \ddots & \ddots & -1 \\
 &  &  &  & -1 & 2
\end{bmatrix},
\end{equation}
and the remaining blocks (of sizes $(\lambda_i - 1) \times (\lambda_i - 1)$ for $i = 1, \ldots, k$) are tridiagonal matrices of the form
\begin{equation}
\label{eq:block weight 0}
B_{\lambda_i-1} = 
\begin{bmatrix} 
2 & -1 &  &  &  &  \\
-1 & 2 & -1 &  &  &  \\
 & -1 & 2 & -1 &  &  \\
 &  & -1 & 2  & \ddots &  \\
 &  &  & \ddots & \ddots & -1\\
 &  &   &  & -1 & 2
\end{bmatrix}.
\end{equation}
The determinants of such tridiagonal matrices satisfy standard recurrence relations (that can be derived by cofactor expansions), and these recurrences yield 
\[
\det B_{\lambda_i-1} = \lambda_i \quad \text{and} \quad \det A_{\lambda_0} = (1-\xi)(\det B_{\lambda_0-1} - \det B_{\lambda_0-2}) = 1-\xi.
\]
Thus $\GD({\bm \rho}_g) =  (1 - \xi) \prod_{i = 1}^k \lambda_i$.

Next we consider the denominator $\GD(\bm{\rho_{\ttt}}\cup\bm{\rho}_g)$ associated to a relative generating set $\ttt$.  
By definition, $\bm{\rho_{\ttt}}\cup\bm{\rho}_g$ is a set of roots associated to a good generating set of reflections for $W$.  Place an arbitrary order on this set of reflections; by Proposition~\ref{Prop: prod_t_i is q-Cox}, the resulting $n$-tuple $\ttt'$ is a reduced factorization of a quasi-Coxeter element $g'$ for $W$.  By Proposition~\ref{prop:RGS characterization}, $\ttt$ consists of $k$ transposition-like reflections, so $\wt(g')=\wt(g)$.  By conjugation-invariance, we may assume that $g' = [(1\cdots n); (0, \ldots, 0, c)]$ for $c = \wt(g)$ (a generator of $\ZZ/m\ZZ$).  By Proposition~\ref{Prop: LW G(m, p, n) transitivity}, all reduced factorizations of $g'$ lie in a single Hurwitz orbit, so $\ttt'$ is Hurwitz-equivalent to the factorization  $\ttt'' = \left( [\id, (c, 0, \ldots, 0)] , (12), \ldots, (n-1 \; n)\right)$ of $g'$.  By Proposition~\ref{prop: GD invariant to Hurwitz moves}, $\GD({\bm \rho}_{\ttt} \cup {\bm \rho}_g) = \GD({\bm \rho}_{\ttt''})$.  The associated Gram matrix is $n \times n$ and of the same form as in Equation~\eqref{eq:block}, so has determinant $1 - \xi$ (where again $\xi = \zeta_m^c$). Therefore we have that $\GD(\bm{\rho}_g)/(\GD(\bm{\rho_{ t}} \cup \bm{\rho}_g)) = \prod_{i=1}^k \lambda_i$.

From the preceding two paragraphs, we have that the summands appearing in the right side of \eqref{eq:volume theorem} for the given $W$, $g$ are constant---specifically, they all equal $\prod_{i = 1}^k \lambda_i$.  Therefore,
\begin{equation}
\label{eq: rgs multiplicity case 1(a)}
\sum_{ \ttt \in \RGS(G(m,1,n), g)} \frac{ \GD(\bm{\rho}_g)}{\GD(\bm{\rho_{\ttt}}\cup\bm{\rho}_g)} = \# \RGS(G(m,1,n), g) \cdot \prod_{i=1}^k \lambda_i.
\end{equation}
By Theorem~\ref{thm:count rgs}(i),
\[
\# \RGS(G(m, 1, n), g) = m^{k} \cdot n^{k - 1} \cdot \prod_{i = 0}^k \lambda_i,
\]
and therefore \eqref{eq: rgs multiplicity case 1(a)} becomes
\begin{equation*}
\sum_{ \ttt \in \RGS(G(m,1,n), g)} \frac{ \GD(\bm{\rho}_g)}{\GD(\bm{\rho_{\ttt}}\cup\bm{\rho}_g)} =  m^{k} \cdot n^{k - 1} \cdot \lambda_0 \cdot \prod_{i = 1}^k \lambda_i^2.
\end{equation*}
This leaves only to compute the prefactors on the right side of \eqref{eq:volume theorem}.  By Theorem~\ref{thm:parabolic everything in G(m, p, n)}, in the generalized cycle decomposition $g = g_0 \cdots g_k$, the factors are precisely the individual cycles.
By Remark~\ref{rem:Fred extends} and Corollary~\ref{cor: Fred for infinite families}(ii) and~(i), we have 
\[
\frac{\Fred(g_0)}{\lR(g_0)!} = \frac{\lambda_0^{\lambda_0}}{\lambda_0!}  \qquad \text{and} \qquad \frac{\Fred(g_i)}{\lR(g_i)!} = \frac{\lambda_i ^{\lambda_i - 2}}{(\lambda_i - 1)!} \text{ for $i = 1, \ldots, k$}.
\]
Putting everything together, we have that the right side of Equation~\eqref{eq:volume theorem} for the selected element $g$ is
\[
(n + k)! \cdot \frac{\lambda_0^{\lambda_0}}{\lambda_0!} \cdot \prod_{i = 1}^k \frac{\lambda_i ^{\lambda_i - 2}}{(\lambda_i - 1)!} \cdot \left(m^{k} \cdot n^{k - 1} \cdot \lambda_0 \cdot \prod_{i = 1}^k \lambda_i^2\right)
=
(n + k)! \cdot m^k \cdot n^{k - 1} \cdot \prod_{i = 0}^k \frac{\lambda_i ^{\lambda_i}}{(\lambda_i - 1)!}.
\]
Comparing with \eqref{eq:case 1(a) LHS} gives the result in this case.
\qed

\bigskip

\noindent {\bf Case 1(b):} Next, suppose $g$ is a quasi-Coxeter element for the parabolic subgroup $\Symm_{\lambda_1} \times \cdots \times \Symm_{\lambda_k}$ of $W = G(m, 1, n)$.  
By Theorem~\ref{thm:parabolic everything in G(m, p, n)}(i), $g$ has $k$ cycles, all of color $0$.  Thus $\gcd(\col(w), m) = m \neq 1$ and we have by Proposition~\ref{prop:formula for Ftr in the combinatorial case} that
\begin{equation}
 \label{eq:case 1(b) LHS}
\Ftr_W(g) = \frac{n^2 (n+k)(n+k-1)m^{k} }{2} \cdot \frac{\varphi(m)}{m} \cdot H_0(\lambda) =  (n + k)! \cdot \frac{\varphi(m)}{2} \cdot m^{k - 1} \cdot n^{k - 1} \cdot \prod_{i = 1}^k \frac{\lambda_i^{\lambda_i}}{(\lambda_i - 1)!},
\end{equation}
where $\varphi$ is the Euler totient function and where again we use the formula of Theorem~\ref{thm:S_n genus 0} for the Hurwitz number $H_0(\lambda)$.

Next, we consider the right side of \eqref{eq:volume theorem} for the given element $g$.  We begin by considering the sum over  $\RGS(W, g)$.

As in Case 1(a), by using the invariance results of Section~\ref{sec:invariance} we may as well replace $g$ with the permutation $(1 \cdots \lambda_1)((\lambda_1 + 1) \cdots (\lambda_1 + \lambda_2)) \cdots ((n - \lambda_k + 1) \cdots n)$ in $\Symm_n$, and select as its fixed factorization the product $(12)(23)\cdots (\lambda_1-1 \ \lambda_1) \cdot ((\lambda_1 + 1)(\lambda_1 + 2)) \cdots (n - 1 \ n)$ of adjacent transpositions.  Then, following the same analysis as in the previous case, the Gram matrix of ${\bm \rho}_g$ is block-diagonal, and its $k$ blocks are of the form \eqref{eq:block weight 0}, with size $(\lambda_i - 1) \times (\lambda_i - 1)$ and determinant $\lambda_i$ for $i = 1, \ldots, k$. Consequently $\GD({\bm \rho}_g) = \prod_{i = 1}^k \lambda_i$.

Next we consider the denominator $\GD({\bm\rho}_{\ttt} \cup {\bm\rho}_g)$ associated to a relative generating set $\ttt$.  By Proposition~\ref{prop:RGS characterization}(ii), $\ttt$ consists of a rooted tree relative to the $k$ cycles of $g$, one of whose factors is a diagonal reflection whose color $c$ is a primitive generator for $\ZZ/m\ZZ$.  Just as in Case 1(a), when such a set of reflections is combined with the set of reflections in our fixed factorization, the Gram determinant of the associated roots is uniquely determined by the eigenvalue $\xi = \zeta_m^c$ of the diagonal reflection, with value $\GD({\bm\rho}_{\ttt} \cup {\bm\rho}_g)=1 - \xi$.

Unlike in Case 1(a), the summands on the right side of \eqref{eq:volume theorem} are not constant.  However, as observed in \cite[Rem.~7.5]{DLM2}, each of the $\varphi(m)$ primitive $m$th roots of unity occurs equally often as an eigenvalue for the diagonal reflection in the elements of $\RGS(W, g)$ (see also Section~\ref{Sec: GD statistic and RGS for Gmpn}).  Combining this with Theorem~\ref{thm:count rgs}(ii) gives
\begin{align}
\label{eq: rgs multiplicity case 1(b)}
\sum_{ \ttt \in \RGS(W, g)} \frac{ \GD(\bm{\rho}_g)}{\GD(\bm{\rho_{\ttt}}\cup\bm{\rho}_g)} &= 
\frac{\# \RGS(W,g)}{\varphi(m)}\sum_{\xi \textrm{ prim.}} \frac{\prod_{i=1}^k \lambda_i}{1-\xi}
\\
\notag
&= 
 m^{k - 1} \cdot n^{k - 1} \cdot \prod_{i=1}^k \lambda_i^2 \cdot \left( \sum_{\xi \textrm{ prim.}} \frac{1}{1-\xi}\right).    
\end{align}

Since (by Theorem~\ref{thm:parabolic everything in G(m, p, n)}) each $g_i$ is a cycle of color $0$, we have by Remark~\ref{rem:Fred extends} and Corollary~\ref{cor: Fred for infinite families}(i) that $\dfrac{\Fred(g_i)}{\lR(g_i)!} = \dfrac{\lambda_i ^{\lambda_i - 2}}{(\lambda_i - 1)!}$ for $i = 1, \ldots, k$.  Moreover, by Proposition~\ref{prop. primitive root identity 1} (stated and proved below in Appendix~\ref{appendix}), we have $\sum_{\xi \textrm{ prim.}} 1/(1-\xi) = \varphi(m)/2$.  Putting this all together, we have that the right side of \eqref{eq:volume theorem} is
\[
(n + k)! \cdot m^{k - 1} \cdot n^{k - 1} \cdot \prod_{i = 1}^k \frac{\lambda_i ^{\lambda_i}}{(\lambda_i - 1)!}   \cdot \frac{\varphi(m)}{2}.
\]
Comparing with \eqref{eq:case 1(b) LHS} gives the result in this case.
\qed

\subsection{The combinatorial family \texorpdfstring{$G(m, m, n)$}{G(m, m, n)} with \texorpdfstring{$m > 1$}{m > 1}}
\label{sec: proof Gmmn}

In this section, we prove Theorem~\ref{volume theorem} for the group $W := G(m, m, n)$.
Let $g$ be a parabolic quasi-Coxeter element in $W$.  According to Theorem~\ref{thm:parabolic everything in G(m, p, n)}, there are two cases: the parabolic subgroup to which $g$ is associated is either of the form $G(m, m, \lambda_0) \times \Symm_{\lambda_1} \times \cdots \times \Symm_{\lambda_k}$ or is a Young subgroup $\Symm_{\lambda_1} \times \cdots \times \Symm_{\lambda_k}$ for the symmetric group $\Symm_n \subset W$.

\medskip

\noindent {\bf Case 2(a):} 
Suppose that $g$ is a quasi-Coxeter element for the parabolic subgroup $G(m, m, \lambda_0) \times \Symm_{\lambda_1} \times \cdots \times \Symm_{\lambda_k}$, with $\lambda_0 \geq 2$. 

In this case, by Theorem~\ref{thm:parabolic everything in G(m, p, n)}(iv), $g$ has $k + 2$ cycles---$k$ cycles of color $0$ (having lengths $\lambda_1, \ldots, \lambda_k$) and two cycles whose colors generate $\ZZ/m\ZZ$ and sum to $0$ (having lengths $a$ and $b$, with $a + b = \lambda_0$)---and the gcd of the cycle colors with $m$ is $1$. Thus by the $d=1$ case of Proposition~\ref{prop:formula for Ftr in the combinatorial case} and the formula for the Hurwitz numbers in Theorem~\ref{thm:S_n genus 0} we have
\begin{equation}
 \label{eq:case 2(a) LHS}
\Ftr_W(g) =  m^{(k+2) - 1} \cdot H_0(a, b, \lambda_1, \ldots, \lambda_k)
         = (n+k)! \cdot m^{k+1} \cdot n^{k-1} \cdot  \frac{a^ab^b}{(a-1)!(b-1)!} \cdot \prod_{i=1}^k \frac{\lambda_i^{\lambda_i}}{(\lambda_i-1)!}.
\end{equation}

Next, we consider the right side of \eqref{eq:volume theorem} for the given element $g$, again taking advantage of the invariances allowed by Section~\ref{sec:invariance}.
The conjugacy class of $g$ contains an element that is the product $g_0 g_1\cdots g_k$ where
\[
g_0 =  (12) \cdot (23) \cdots (a \; a+1) \cdot [(a \; a+1); c] \cdot (a + 1 \; a + 2) \cdots (\lambda_0 - 1 \; \lambda_0),
\]
where $c$ is the color of the $a$-cycle  in $W$ (a primitive generator of $\ZZ/m\ZZ$), and each of the $g_i$ is a product of adjacent transpositions, which together form a generating set of the Young subgroup $\Symm_{\lambda_1} \times \cdots \times \Symm_{\lambda_k}$.  As before, we may choose the roots and coroots of the adjacent transpositions to be $\rho_i = \widecheck{\rho}_i = (0, \ldots, 0, 1, -1, 0, \ldots, 0)$ while the anomalous factor $[(a - 1 \; a); c]$ (of non-unit eigenvalue $-1$) has root and coroot $\rho_a = \widecheck{\rho}_a = (0, \ldots, 0, 1, -\xi, 0, \ldots, 0)$ where $\xi := \zeta_m^c$.  The associated Gram matrix is block-diagonal.  The first ($\lambda_0 \times \lambda_0$) block $D_{\lambda_0}$ is of one of the following forms: if $\lambda_0 = 2$ (so $a = b = 1$) then $D_2=\begin{bmatrix}
2 & 1 + \overline{\xi} \\
1+\xi & 2
\end{bmatrix}$, and otherwise $D_{\lambda_0}$ is one of 
\begin{equation}
\label{eq:mmn block}
\includegraphics[align=c]{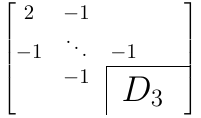},
\quad
\includegraphics[align=c]{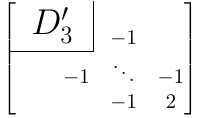},
\quad \text{or} \quad
\includegraphics[align=c]{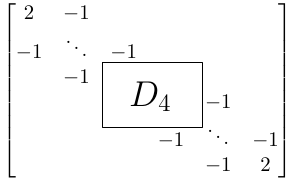},
\end{equation}
where
\[
D_3 = \begin{bmatrix}
2 & -1 & -1 \\
-1 & 2 & 1 + \overline{\xi}\\
-1 & 1+\xi & 2
\end{bmatrix}, \quad
D_3' = \begin{bmatrix}
2 & 1+\overline{\xi} & -1 \\
1+ \xi & 2 & -\xi\\
-1 & -\overline{\xi} & 2
\end{bmatrix}, \;\; \text{and} \;\;
D_4 = \begin{bmatrix} 
2 & -1 & -1& \\
-1 & 2 & 1 + \overline{\xi} & - 1\\
-1 & 1 + \xi & 2 & -\xi\\
& -1 & -\overline{\xi} & 2
\end{bmatrix},
\]
depending on whether $a = \lambda_0 - 1$ (and $b = 1$), $a = 1$ (and $b = \lambda_0 - 1$), or $a$ and $b$ are both at least $2$.
We claim that the determinant of such blocks is always $\det D_{\lambda_0} =2 - \xi - \overline{\xi}$. Indeed, one can check directly that $\det D_2=\det D_3 = \det D'_3=\det D_4 = 2-\xi-\overline{\xi}$. Also, if the block $D_{\lambda_0}$ starts with \({\small  \begin{array}{cc} 2&-1\\ -1 & \end{array}}\) (respectively, ends with  \({\small  \begin{array}{cc} &-1\\ -1 & 2\end{array}}\)), doing a cofactor expansion along the first (respectively, last) row and using induction gives \[
\det D_{\lambda_0} = 2 \cdot \det D_{\lambda_0-1} -1 \cdot \det D_{\lambda_0-2} = (2-1)(2-\xi-\overline{\xi}) = 2-\xi-\overline{\xi},
\]
where $D_{\lambda_0-1}$ and $D_{\lambda_0-2}$ are the respective $(\lambda_0-1)\times (\lambda_0-1)$ and $(\lambda_0-2)\times (\lambda_0-2)$ submatrices of $D_{\lambda_0}$.
The remaining blocks (of sizes $(\lambda_i - 1) \times (\lambda_i - 1)$ for $i = 1, \ldots, k$) are of the same form as the block $B_{\lambda_i-1}$ in \eqref{eq:block weight 0}, with determinants $\lambda_1, \ldots, \lambda_k$.  Thus $\GD({\bm\rho}_g) = (2 -  \xi - \overline{\xi})\lambda_1 \cdots \lambda_k$.

Next we consider the denominator $\GD({\bm\rho}_{\ttt} \cup {\bm\rho}_g)$ associated to a relative generating set $\ttt$.  By Proposition~\ref{prop:RGS characterization}(i), $\ttt$ consists of a tree relative to the partition $\Pi_g$ of $\{1, \ldots, n\}$ into the $k + 1$ generalized cycles of $g$.  Thus, for any such $\ttt$, the unicycle associated to ${\bm\rho}_{\ttt} \cup {\bm\rho}_g$ has the same graph cycle as the unicycle component of ${\bm\rho}_g$. Consequently all $\ttt$ give the same value $\delta({\bm\rho}_{\ttt} \cup {\bm\rho}_g) = c$ for the quantity $\delta$ defined in Section~\ref{sec:rgs}. Let $\ttt'$ be the tuple of roots that results from placing an arbitrary order on the set ${\bm\rho}_{\ttt} \cup {\bm\rho}_g$. By Proposition~\ref{Prop: prod_t_i is q-Cox}, $\ttt'$ is a reduced factorization of a quasi-Coxeter element $g'$ for $W$.  By Remark~\ref{rem: Shi's delta}, the two cycles of $g'$ have colors $c$ and $-c$ (the nonzero colors of cycles of $g$), and by conjugation-invariance we may assume $g'$ has underlying permutation $(1 \cdots a')((a' + 1) \cdots n)$ for some $a'$.  By Proposition~\ref{Prop: LW G(m, p, n) transitivity}, all reduced factorizations of $g'$ lie in a single Hurwitz orbit, so $\ttt'$ is Hurwitz-equivalent to $\ttt'' = \left((12), \ldots, (a' \; a'+1), [(a' \; a'+1); c], (a'+1 \; a' + 2), \ldots, (n - 1 \; n)\right)$.  By Proposition~\ref{prop: GD invariant to Hurwitz moves}, $\GD({\bm\rho}_{\ttt} \cup {\bm\rho}_g) = \GD({\bm\rho}_{\ttt''})$.  The associated Gram matrix is $n\times n$ and of the same form as one of the matrices in \eqref{eq:mmn block}, so has determinant $2 - \xi - \overline{\xi}$ (where again $\xi = \zeta_m^c$).

From the preceding two paragraphs, as in Case 1(a), we have that the summands appearing in the right side of \eqref{eq:volume theorem} for the given $W$ and $g$ are constant---specifically, they all equal $\prod_{i = 1}^k \lambda_i$.  Therefore, 
\begin{equation} \label{eq: rgs multiplicity case 2(a)}
\sum_{ \ttt \in \RGS(G(m,m,n), g)} \frac{ \GD(\bm{\rho}_g)}{\GD(\bm{\rho_{\ttt}}\cup\bm{\rho}_g)} = \# \RGS(G(m,m,n), g) \cdot \prod_{i=1}^k \lambda_i.
\end{equation}
By Theorem~\ref{thm:count rgs}(i),
\[
\# \RGS(G(m, m, n), g) = m^{k} \cdot n^{k - 1} \cdot \prod_{i = 0}^k \lambda_i.
\]
This leaves only to compute the prefactors on the right side of \eqref{eq:volume theorem}.  By Remark~\ref{rem:Fred extends} and Corollary~\ref{cor: Fred for infinite families}(iii) and (i), we have (recalling $\lambda_0=a+b$) that
\[
\frac{\Fred(g_0)}{\lR(g_0)!} = \frac{m a^a b^b}{\lambda_0(a-1)!(b-1)!}  \qquad \text{and} \qquad \frac{\Fred(g_i)}{\lR(g_i)!} = \frac{\lambda_i ^{\lambda_i - 2}}{(\lambda_i - 1)!} \text{ for $i = 1, \ldots, k$}.
\]
Putting everything together, we have that the right side of Equation~\eqref{eq:volume theorem} for the selected element $g$ is
\begin{multline*}
(n + k)! \cdot \frac{m a^a b^b}{\lambda_0(a-1)!(b-1)!} \cdot \prod_{i = 1}^k \frac{\lambda_i ^{\lambda_i - 2}}{(\lambda_i - 1)!} \cdot \left(m^{k} \cdot n^{k - 1} \cdot \lambda_0 \cdot \prod_{i = 1}^k \lambda_i^2\right)
\\
= (n + k)! \cdot m^{k + 1} \cdot n^{k - 1} \cdot \frac{a^ab^b}{(a - 1)! (b - 1)!} \cdot \prod_{i = 1}^k  \frac{\lambda_i^{\lambda_i}}{(\lambda_i - 1)!}.
\end{multline*}
Comparing with \eqref{eq:case 2(a) LHS} gives the result in this case.
\qed

\medskip

\noindent {\bf Case 2(b):} 
Finally, suppose that $g$ is a quasi-Coxeter element for the parabolic subgroup $\Symm_{\lambda_1} \times \cdots \times \Symm_{\lambda_k}$ of $G(m, m, n)$.  

By Theorem~\ref{thm:parabolic everything in G(m, p, n)}(i), $g$ has $k$ cycles, all of color $0$, and so the gcd of the cycle colors and $m$ is equal to $m$.  Then by  Proposition~\ref{prop:formula for Ftr in the combinatorial case} we have
\begin{equation}
    \label{eq:case 2(b) LHS}
\Ftr_W(g) = m^{k+1} \cdot H_1(\lambda_1, \ldots, \lambda_k) \cdot \frac{1}{m^2}\cdot \sum_{r \mid m} \mu(m / r) \cdot r^{2}.
\end{equation}
The quantity $H_1(\lambda_1, \ldots, \lambda_k)$ is given by the following explicit formula.

\begin{theorem}[{Goulden--Jackson \cite{GJ99}; Vakil \cite{V01}}]
\label{GJVGJV}
  For $\lambda=(\lambda_1,\ldots,\lambda_k)$, the number of genus-$1$ transitive transposition factorizations in $\Symm_n$ of a permutation of cycle type $\lambda$ is
  \[
H_1(\lambda) = \frac{1}{24} (n+k)! \left(\prod_{i=1}^k
  \frac{\lambda_i^{\lambda_i}}{(\lambda_i-1)!}\right) \left( n^k -
  n^{k-1}  - \sum_{i=2}^k (i-2)! \cdot e_i(\lambda)\cdot  n^{k-i}\right),
\]
where $e_i$ denotes the $i$th elementary symmetric function.
\end{theorem}

Next, we consider the right side of \eqref{eq:volume theorem} for the given element $g$.  We begin by considering the sum over $\RGS(W, g)$.

Taking advantage of the freedom allowed by Section~\ref{sec:invariance}, we may as well assume that
\[
g = (1 \cdots \lambda_1)((\lambda_1 + 1) \cdots (\lambda_1 + \lambda_2)) \cdots ((n - \lambda_k + 1) \cdots n)
\]
and that its fixed reduced factorization is a product of adjacent transpositions.  Then, as in Case~1(b), the associated Gram matrix is block-diagonal, with blocks of the form $B_{\lambda_i-1}$ in \eqref{eq:block weight 0}, with respective determinant $\lambda_i$, and so $\GD({\bm\rho}_g) = \prod_{i = 1}^k \lambda_i$.

Next we consider the denominator $\GD({\bm\rho}_{\ttt} \cup {\bm\rho}_g)$ associated to a relative generating set $\ttt$.  By Proposition~\ref{prop:RGS characterization}(iii), $\ttt$ corresponds to a unicycle relative to the $k$ cycles of $g$.  We then combine the set $\ttt$ with the set of reflections in our fixed factorization of $g$. Just as in Case 2(a), the combined set of reflections corresponds to a unicycle, with Gram determinant $\GD({\bm\rho}_{\ttt} \cup {\bm\rho}_g) = 2 - \xi - \overline{\xi}$, where $\xi = \zeta_m^{\delta({\bm\rho}_{\ttt} \cup {\bm\rho}_g)}$ and $\delta$ is as defined in Section~\ref{sec:rgs}.

As in Case 1(b), the summands on the right side of \eqref{eq:volume theorem} are not constant, and so we split the sum according to the value of the summand. As observed in \cite[Rem.~7.5]{DLM2}, each primitive generator of $\ZZ/m\ZZ$ occurs equally often as the value of $\delta({\bm\rho}_{\ttt} \cup {\bm\rho}_g)$.  It follows that
\begin{equation}
    \label{eq:case 2(b) sum} 
\sum_{ \ttt \in \RGS(W, g)} \frac{ \GD(\bm{\rho}_g)}{\GD(\bm{\rho_{\ttt}}\cup\bm{\rho}_g)} 
= \frac{\#\RGS(W, g)}{\varphi(m)} \cdot 
\sum_{\xi \textrm{ prim.}} \frac{\prod_{i=1}^k \lambda_i}{2-\xi - \overline{\xi}}.
\end{equation}
By Theorem~\ref{thm:count rgs}(iii),
\begin{align*}
\frac{\# \RGS(W, g)}{\varphi(m)}  &= \frac{m^{k - 1}}{2} \prod_{i = 1}^k \lambda_i  \cdot
\Big(n^k - n^{k-1} -\sum_{j=2}^k  (j-2)! \cdot n^{k-j} e_j(\lambda)\Big).
\end{align*}
Substituting this into \eqref{eq:case 2(b) sum} gives
\[
\sum_{ \ttt \in \RGS(W, g)} \frac{ \GD(\bm{\rho}_g)}{\GD(\bm{\rho_{\ttt}}\cup\bm{\rho}_g)} 
= \frac{m^{k - 1}}{2} \prod_{i = 1}^k \lambda_i^2  \cdot
\Big(n^k - n^{k-1} -\sum_{j=2}^k  (j-2)! \cdot n^{k-j} e_j(\lambda)\Big) \cdot 
\sum_{\xi \textrm{ prim.}} \frac{1}{2-\xi - \overline{\xi}}.
\]
Since (by Theorem~\ref{thm:parabolic everything in G(m, p, n)}(iii)) each $g_i$ is a cycle of color $0$, we have by Remark~\ref{rem:Fred extends} and Corollary~\ref{cor: Fred for infinite families}(i) that $\dfrac{\Fred(g_i)}{\lR(g_i)!} = \dfrac{\lambda_i ^{\lambda_i - 2}}{(\lambda_i - 1)!}$ for $i = 1, \ldots, k$. Next, by Proposition~\ref{prop: primitive root identity 2} (stated and proved below in Appendix~\ref{appendix}), we have $\sum_{\xi \textrm{ prim.}} 1/(2-\xi - \overline{\xi}) = \frac{1}{12} \sum_{r \mid m} \mu(m/r) \cdot r^2$.  Putting this all together, we have that the right side of \eqref{eq:volume theorem} is
\begin{multline*}
(n + k)! \cdot  \prod_{i=1}^k \frac{\lambda_i^{\lambda_i - 2}}{(\lambda_i - 1)!} \cdot \left(
 \frac{m^{k - 1}}{2} \prod_{i=1}^k \lambda_i^2 \cdot 
\Big(n^k - n^{k-1} -\sum_{j=2}^k  (j-2)!n^{k-j} e_j(\lambda)\Big) \cdot
\frac{1}{12} \sum_{r \mid m} \mu(m/r) \cdot r^2\right) \\
=(n + k)! \cdot  \frac{m^{k - 1}}{24} \cdot  \prod_{i=1}^k
\frac{\lambda_i^{\lambda_i}}{(\lambda_i - 1)!} \cdot \Big(n^k -
n^{k-1} -\sum_{j=2}^k  (j-2)!n^{k-j} e_j(\lambda)\Big) \cdot
\Big(\sum_{r \mid m} \mu(m/r) \cdot r^2\Big) \\
= m^{k-1}H_1(\lambda) \cdot \sum_{r \mid m} \mu(m/r) \cdot r^2,
\end{multline*}
where in the last step we use the formula for $H_1(\lambda)$ in Theorem~\ref{GJVGJV}. Comparing with \eqref{eq:case 2(b) LHS} gives the result in this case.
\qed

\subsection{Weyl groups via a cut-and-join-type recursion}
\label{Sec: cut-and-join counting}

In this section, we give a generalization (Theorem~\ref{thm:cut and join}) of the cut-and-join recurrence relation for the enumeration of full factorizations in $\Symm_n$ of Goulden and Jackson \cite{GJ97}.\footnote{This is separate from the work of \cite{fesler2023}, where a cut-and-join recurrence is developed for reflection factorizations in types $B$, $D$, but without addressing the transitivity/fullness property.}  While this result is interesting on its own terms, its main application in the present paper is to prove Theorem~\ref{Thm: Weyl case} for the group $E_8$ in the next subsection.

We begin with a lemma that records some information about prefixes of minimum-length full reflection factorizations of the identity in a real reflection group.

\begin{lemma}
\label{Lem: stat on (W_i,g_i)}
In a real reflection group $W$, consider any minimum-length full reflection factorization $t_1\cdot t_2\cdots t_{2n}=\id$ of the identity.
Then for any $i=1,\dots, 2n$, the pair $(g_i,W_i)$ with $W_i=\langle t_1,\cdots ,t_i\rangle$ and $g_i=t_1\cdots t_i$ has the following properties:
\begin{enumerate}
\item $2 \cdot \rank(W_i) - \lR(g_i)=i$,
\item the factorization $g_i = t_1\cdots t_i$ is a minimum-length full reflection factorization in $W_i$, and
\item $g_i$ is a parabolic quasi-Coxeter element of $W_i$.
\end{enumerate}
\end{lemma}
\begin{proof}
For part (1), let $m_i:=2\rank(W_i)-\lR(g_i)$.  Since $g_{i + 1} = g_i t_{i + 1}$ and $W$ is real, we have $\lR(g_{i + 1}) = \lR(g_i) \pm 1$.  Notice next that $\rank(W_i)\leq \rank(W_{i+1}) \leq \rank(W_i)+1$.  Therefore, if $\lR(g_{i+1})=\lR(g_i)+1$, then $m_{i+1}\leq m_i +1$. On the other hand, if $\lR(g_it_{i+1})=\lR(g_i) - 1$, then by Lemma~\ref{Lem: g->gt} we have that $V^{t_{i + 1}} \supseteq V^{g_i} \supseteq V^{W_i}$.  Since $W_{i + 1} = \langle W_i, t_{i + 1}\rangle$, it follows that $V^{W_{i + 1}} = V^{W_i}$ and so that  $\rank(W_{i+1})=\rank(W_i)$.  Thus in this case $m_{i+1}=m_i + 1$. Since we always have $m_{i+1}\leq m_i +1$ and since $m_0=0$ and $m_{2n}=2n$, we must have $m_i=i$ for all $i=1, \dots, 2n$.

For part (2), the factorization is full by definition of $W_i$, and therefore $\ltr[W_i](g_i) \leq i$.  On the other hand, combining Corollary~\ref{Corol: lr(g)+ltr(g)>=2m} with part (1), we have that $\ltr[W_i](g_i) \geq 2 \rank(W_i) - \lR[W_i](g_i) = i$.

Finally, part (3) is a direct consequence of parts (1) and (2), the characterization of parabolic quasi-Coxeter elements from Theorem~\ref{Prop: characterization of pqCox}[(i)$\Leftrightarrow$(iv)], and the fact that in a real reflection group $\lR(g_i)=\codim(V^{g_i})=\lR[W_i](g_i)$ (see Section~\ref{sec:reflen}).
\end{proof}

The name ``cut-and-join'' comes from the fact that multiplying a permutation $g \in \Symm_n$ by a transposition $t$ either \emph{cuts} a cycle of $g$ into two cycles, or \emph{joins} two cycles of $g$ into a single cycle.  In our setting, $\Symm_n$ is replaced by a real reflection group $W$, and the transposition $t$ is replaced by a reflection.

\begin{theorem}[cut-and-join recursion for real groups]
\label{thm:cut and join}
Let $W$ be a real reflection group of rank $n$, $g\in W$ a parabolic quasi-Coxeter element, and $W_g$ the parabolic closure of $g$ in $W$. Then the number $\Ftr_W(g)$ of minimum-length full reflection factorizations of $g$ satisfies
\begin{equation}
    \label{cut-and-join-equation for real groups}
\Ftr_W(g)=\Biggl(\,\sum_{t\in W_g}\sum_{\substack{gt \in W'\leq W, \\ \text{ condition (1)}}}\Ftr_{W'}(gt)\Biggr)\,+\,\Biggl(\,\sum_{t\notin W_g}\sum_{\substack{gt \in W'\leq W, \\ \text{condition (2)}}}\Ftr_{W'}(gt)\Biggr).
\end{equation}
The outer summations are over reflections $t$, while the inner summations are over the following groups $W'$:
\begin{enumerate}
\item In the first sum ($t\in W_g$), $W'$ is a parabolic subgroup of $W$ of rank $n-1$ such that $\langle W',t\rangle=W$.
\item In the second sum  ($t\notin W_g$), $W'$ is a rank-$n$ reflection subgroup of $W$ (including the case $W' = W$) such that $\langle W',t\rangle = W$ and $gt$ is parabolic quasi-Coxeter for $W'$. 
\end{enumerate}
Furthermore, the groups $W'$ in the first sum (those satisfying condition (1)) contain $gt$ as a parabolic quasi-Coxeter element.
\end{theorem}
\begin{proof}
Consider the last reflection $t_k$ that appears in some minimum-length full reflection factorization $g = t_1\cdots t_k$. Since $g$ is parabolic quasi-Coxeter, we have by Theorem~\ref{Prop: characterization of pqCox}[(i)$\Leftrightarrow$(iv)] that $k = \ltr(g) = 2n - \lR(g)$. This means in particular that we may extend this factorization of $g$ to a minimum-length full reflection factorization of the identity in $W$ by adjoining a reduced factorization of $g^{-1}$. This puts us in the context and assumptions of Lemma~\ref{Lem: stat on (W_i,g_i)}.

{\bf Assume first that $t_k\in W_g$} and let $W':=\langle t_1,\cdots, t_{k-1}\rangle$.  By Lemma~\ref{Lem: stat on (W_i,g_i)}(2) and the previous paragraph, we have that $t_1\cdots t_{k-1}=gt_k$ is a minimum-length full $W'$-reflection factorization of $gt_k$.  We must check that $W'$ satisfies condition (1). To begin with, by definition of $W'$ we have that $\langle W', t_k\rangle = \langle t_1, \ldots, t_k \rangle = W$. Now, Lemma~\ref{Lem: g->gt}(3) states that $\lR(gt_k)=\lR(g)-1$ and hence by Lemma~\ref{Lem: stat on (W_i,g_i)}(1) we also know that $\rank(W')=n-1$. Then Proposition~\ref{Prop: Taylor_K=<H,t>} implies that $W'$ is parabolic since $W=\langle W',t_k\rangle$ is itself parabolic and of greater rank.

We will show now that if $t\in W_g$, then any group $ W'\leq W$ that satisfies condition (1) with $gt \in W'$ contains $gt$ as a parabolic quasi-Coxeter element. Indeed, by Lemma~\ref{Lem: g->gt}(3) we have that $gt\leq_{\RRR} g$, which means after Theorem~\ref{Prop: characterization of pqCox}[(i)$\Leftrightarrow$(iii)] (and since $\leq_{\RRR}$ is transitive) that $gt$ is a parabolic quasi-Coxeter element for $W$. This in turn means that $gt$ is a parabolic quasi-Coxeter element for any parabolic subgroup $W'\leq W$ that contains it \cite[Cor.~3.5]{DLM2}.

To see that all contributions of the first sum are valid, start with any reflection $t\in W_g$. Consider now any subgroup $W'\leq W$ that satisfies condition (1) with $gt \in W'$, and any minimum-length full $W'$-reflection factorization $t_1\cdots t_r=gt$. From the previous paragraph we have that $gt$ is parabolic quasi-Coxeter in $W'$. Thus since $\rank(W')=n-1$  we have by Theorem~\ref{Prop: characterization of pqCox}[(i)$\Leftrightarrow$(iv)] that $r=2(n-1)-\lR[W'](gt)=2(n-1)-\lR(gt) = 2n - \lR(g) - 1$. Then $t_1\cdots t_r\cdot t$ is a reflection factorization of $g$, of length $2n-\lR(g)$, whose terms generate $W$. That is, it is a minimum-length full reflection factorization of $g$ and thus a valid contribution for $\Ftr_W(g)$.

{\bf Assume now that $t_k\notin W_g$} and let $W':=\langle t_1,\cdots,t_{k-1}\rangle$.  We have by the same lemmata as above that $t_1\cdots t_{k-1}=gt_k$ is a minimum-length full $W'$-reflection factorization, that $\lR(gt_k)=\lR(g)+1$, and hence that $\rank(W')=n$. By Lemma~\ref{Lem: stat on (W_i,g_i)}(3) and the discussion in the first paragraph of the proof, we have that $gt_k$ is parabolic quasi-Coxeter in $W'$.

To see that all contributions of the second sum are valid, start with a reflection $t\notin W_g$ and any reflection subgroup $W'\leq W$ satisfying condition (2) with $gt\in W'$. Since $\rank(W')=n$ and $gt$ is assumed parabolic quasi-Coxeter in $W'$, we have that the full reflection length of $gt$ \emph{with respect to the group} $W'$ is equal to $r=2n-\lR[W'](g)-1=2n-\lR(g)-1=\ltr(g)-1$. Since $\langle W',t\rangle =W$, any minimum-length full $W'$-factorization $t_1\cdots t_r=gt$ gives us a minimum-length full $W$-factorization $t_1\cdots t_r \cdot t=g$. This completes the argument.
\end{proof}

\subsection{The exceptional groups}
\label{sec: proof exceptional}

In this section, we describe the computer-assisted part of the proof of Theorems~\ref{Thm: Weyl case} and~\ref{thm:main} for the exceptional reflection groups. We will use two different approaches: the first is valid in principle for all complex types but is more computationally demanding, and in fact fails for the group $E_8$ due to lack of computing power; the second works only for Weyl types, but is faster, more conceptual, and does in fact succeed to prove the case of $E_8$.

\subsubsection*{Exceptional reflection groups apart from $E_8$}

For the exceptional complex reflection groups, we prove Theorem~\ref{volume theorem} by a computer calculation in SageMath and CHEVIE \cite{chevie,sagemath} that we describe now.  For each exceptional group $W$, we produce a set of conjugacy class representatives of parabolic quasi-Coxeter elements using the characterization in Theorem~\ref{Prop: characterization of pqCox}[(i)$\Leftrightarrow$(ii)], namely, for each conjugacy class, we fix a representative $g$ and one of its reduced reflection factorizations and exhaustively check whether combining it with any $(n - \lR(g))$-element subset of $\RRR$ generates $W$.  For each parabolic quasi-Coxeter representative $g$, we determine $\Fred(g)$ and $\Ftr_W(g)$ as described in \cite[\S3]{DLM1}, first computing the power series 
\[
\FFF_W(g; z) := \sum_{N\geq 0}\#\left\{(t_1,\ldots,t_N)\in\RRR^N: t_1\cdots t_N=g\right\}\cdot \frac{z^N}{N!}
\]
that counts \emph{all} reflection factorizations of $g$ by length using the character table of $W$ and a formula for factorization-counting due to Frobenius (see \cite[Eq.~(3.1)]{DLM1}), and then computing the power series
\begin{equation}
\FFFtr_W(g; z) := \sum_{N\geq 0}\#\left\{(t_1,\ldots,t_N)\in\RRR^N : t_1\cdots t_N=g\text{ and }\langle t_1,\ldots, t_N\rangle=W \right\}\cdot \frac{z^N}{N!}\label{Eq: defn Ftr_W(g;t)}
\end{equation}
that counts full factorizations of arbitrary genus via M\"obius inversion on the lattice of reflection subgroups of $W$.  The number $\Ftr_W(g)$ is the first nonzero coefficient of this series.  The software systems provide us further with the matrix representation of each group element  and from this we can produce the eigenspaces of reflections and thus construct the (complex) root systems. Then the right side of Theorem~\ref{thm:main} can be computed after enumerating all relative generating sets, and we can compare it with $\Ftr_W(g)$.  

The result of the computation just described was to verify Theorem~\ref{thm:main} for all exceptional well generated complex reflection groups, including all Weyl groups, \emph{with the exception of $E_8$}.  The reason we omit $E_8$ is that the lattice of reflection subgroups of $E_8$ is too large to implement with our available computer power, and consequently we cannot compute $\FFFtr_{E_8}(g)$ this way. However, $E_8$ is addressed in the next discussion, which covers all exceptional Weyl groups.

\subsubsection*{Exceptional Weyl groups, including $E_8$} 

For the exceptional Weyl groups, we prove the main theorem (Theorem~\ref{Thm: Weyl case}) using Theorem~\ref{thm:cut and join} and a simultaneous induction on the quantities $\rank(W)$ and $\rank(W)-\lR(g)$.  Indeed, for all pairs $(gt, W')$ that appear on the right side of \eqref{cut-and-join-equation for real groups}, the elements $gt$ are parabolic quasi-Coxeter in $W'$ and we have that either $\rank(W') = \rank(W) - 1$ (in the first sum) or $\rank(W') = \rank(W)$ and
$
\rank(W')-\lR[W'](gt)= \rank(W)-\lR(g) - 1
$
(in the second sum).  The base cases (when $\rank(W) = \lR(g)$) are when $g$ is quasi-Coxeter for $W$; in this case, $W_g = W$, $\RGS(W, g) = \{ \varnothing\}$, the generalized cycle decomposition is the trivial factorization $g_1 = g$, and $\ltr(g) = \lR(g)$.  Thus in this case Equation~\eqref{Eq: weyl thm} reduces to the assertion $\Ftr_W(g) = \Fred(g)$ for quasi-Coxeter $g$, and this is precisely the content of \cite[Thm.~1.2]{BGRW}.  

To complete the inductive proof, it is sufficient (applying Theorem~\ref{thm:cut and join}, using Remark~\ref{Rem: fred(g_i)->fred(g)}, and cancelling common factors) to prove the following recurrence on the cardinalities of relative generating sets $\RGS(W,g)$:
\begin{multline}
\label{eq:LHS cut and join Weyl groups}
\ltr(g)\cdot \Fred(g) \cdot\#\RGS(W,g)\cdot \dfrac{I(W_g)}{I(W)}=\\
=\lR(g)\cdot \Biggl(\, \sum_{t\in W_g}\sum_{\substack{gt \in W'\leq W, \\ \text{ condition (1)}}}\Fred[W'](gt)\cdot \#\RGS(W',gt)\cdot \dfrac{I(W'_{gt})}{I(W')}\Biggl)+ {}\\
{}+\dfrac{1}{\lR(g)+1}\Biggl(\,\sum_{t\notin W_g}\sum_{\substack{gt \in W'\leq W, \\ \text{condition (2)}}}\Fred[W'](gt)\cdot \#\RGS(W',gt)\cdot\dfrac{I(W'_{gt})}{I(W')}\Biggl),
\end{multline}
where the conditions on which groups $W'$ participate in the summation are as in Theorem~\ref{thm:cut and join}.

The advantage of this expression is that it rids us of the burden of calculating the numbers $\Ftr_W(g)$ using the approach of \cite[\S3]{DLM1} (described in detail in the previous subsection). In particular, rather than a complete calculation of the poset of reflection subgroups of $W$ that contain $g$ (and of its M\"obius function), we need only have a fast way of computing the quantities $\#\RGS(W,g)$. To count such relative generating sets we test all $(\op{rank}(W)-\lR(g))$-element sets of reflections of $W$; by Proposition~\ref{Prop: prod_t_i is q-Cox}, such a set belongs to the collection $\RGS(W,g)$ if and only if the product of all its elements and $g$, in any order, is a quasi-Coxeter element of $W$.  We use a similar approach to generate all reflection subgroups $W'\leq W$ that contain $g$.

Following this approach we have verified, using SageMath and CHEVIE \cite{chevie,sagemath}, the equality \eqref{eq:LHS cut and join Weyl groups} for all parabolic quasi-Coxeter elements in all exceptional Weyl groups. For $E_8$, this calculation was run on the RCF computing cluster of the Department of Mathematics and Statistics of UMass Amherst, where it took multiple days.

\section{Further comments}
\label{sec: further}

We end with some final remarks on our main theorems and their proofs, including conjectures, open questions (Questions~\ref{Quest_gen_of_thm}, \ref{question: postivity phi Weyl groups}, and \ref{Quest:overcounting_braids}), and potential connections to other mathematical objects (including the cut-and-join equation, a factorization poset, root systems for complex reflection groups, weighted Hurwitz numbers, the braid group action on factorizations, the ELSV formula, and root systems for complex reflection groups).

\subsection{Some remarks on the proof} \label{Sec: GD statistic and RGS for Gmpn}

In some sense, the combinatorial families $G(m,m,n)$ and $G(m,1,n)$ appear close to the Weyl group case in the context of Theorems~\ref{thm:main} and~\ref{Thm: Weyl case}. Even though their (relative) generating sets do not always carry the same Grammian statistic, there are \emph{equal numbers of sets for each value of the statistic}. In particular, comparing formulas \eqref{eq: rgs multiplicity case 1(a)}, \eqref{eq: rgs multiplicity case 1(b)}, \eqref{eq: rgs multiplicity case 2(a)}, and \eqref{eq:case 2(b) sum} and keeping the notation for the partition $\lambda$ associated with $W_g$, we can write
\[
\sum_{\ttt\in\RGS(W,g)}\dfrac{\GD(\bm\rho_g)}{\GD(\bm\rho_{\ttt}\cup\bm\rho_g)}=\#\RGS(W,g)\cdot \left(\prod_{i=1}^k\lambda_i\right)\cdot\mathfrak{c}_g,
\]
where the number $\mathfrak{c}_g$ is $1$, $1/2$, $1$, or $\psi(m)/12$ in the four cases 1(a), 1(b), 2(a), 2(b), respectively, and $\psi(m)$ is the Dedekind function of Remark~\ref{Rem: Dedekind}. This is not the case for general groups, as we can see already for $H_3$ (in Example~\ref{Ex: H_3 calc}) and $H_4$ (in Remark~\ref{Rem: RGS is not a factor} below).

\begin{remark}\label{Rem: RGS is not a factor}
In general, different (relative) generating sets might have different Grammian statistics and the families of sets with a given statistic might have different sizes. That is, the statement of Theorem~\ref{thm:main} is the best one can hope for, and both the Weyl case of Theorem~\ref{Thm: Weyl case} and the nice behavior of the combinatorial families are special phenomena. The numerology of the group $H_4$ makes this fully apparent; there, the number of generating sets and the Grammian statistic are given by
\[
\#\RGS(H_4,\id)=355800=2^3\cdot 3\cdot 5^2\cdot 593\qquad \text{and}\qquad \sum_{\ttt\in\RGS(H_4,\id)}\dfrac{1}{\GD(\bm\rho_{\ttt})}=594300=2^2 \cdot 3 \cdot 5^2 \cdot 7\cdot 283.
\]
There seems to be no chance that the cardinality $\#\RGS(W,g)$ may appear as a factor in a reasonable reformulation of our main Theorem~\ref{thm:main}.
\end{remark}

\subsection{Remarks on the cut-and-join recurrence}
\label{sec:poset}

In Section~\ref{Sec: cut-and-join counting}, we used a cut-and-join recurrence as part of the proof of the main theorem. The usual formulation \cite[Lem.~2.2]{GJ97} of the cut-and-join equations for genus-$0$ Hurwitz numbers $H_0(\lambda)$ in the symmetric group yields a system of partial differential equations with a unique solution.  Having a generating series that ranges over all symmetric groups $\Symm_n$ allowed Goulden--Jackson to use Lagrange inversion to produce that unique solution.  In order to do something similar in our setting, it would be necessary to have separate variables keeping track of each type of generalized cycle (as in \S\ref{sec:parabolic}). This is possible but it appears unlikely it would be fruitful (because, for instance, there is no easy way to place all complex reflection groups in a single generating function). 

We mention another perspective on our generalization of the cut-and-join recurrence. In the case that $g = c$ is a Coxeter element, the reduced reflection factorizations of $c$ (which are also its minimum-length full reflection factorizations) are in natural correspondence with the maximal chains in the \emph{lattice of $W$-noncrossing partitions $NC(W)$} (see \cite{armstrong}).  Lemma~\ref{Lem: stat on (W_i,g_i)} suggests an analogous construction: consider the poset whose elements are the pairs $(g_i, W_i)$ that appear in Lemma~\ref{Lem: stat on (W_i,g_i)} with the natural order relation that $(g_1, W_1) \leq (g_2, W_2)$ if there is a tuple $(t_1, \ldots, t_k)$ of $k:=\ltr[W_2](g_2)-\ltr[W_1](g_1)$ reflections such that $g_2 = g_1 \cdot t_1 \cdots t_k$ and $W_2 = \langle W_1, t_1, \ldots, t_k\rangle$.  Then its maximal chains correspond precisely to the minimum-length full reflection factorizations of the identity; see Figure~\ref{fig: poset factorizations B_2} for the case $W=B_2$.  This poset seems like an interesting object in its own right, and worthy of further study.

\begin{figure}[ht]
    \includegraphics[scale=0.8]{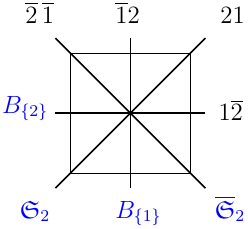}
    \hspace{-50pt}
    \qquad \qquad  \quad
    \includegraphics[scale=0.6]{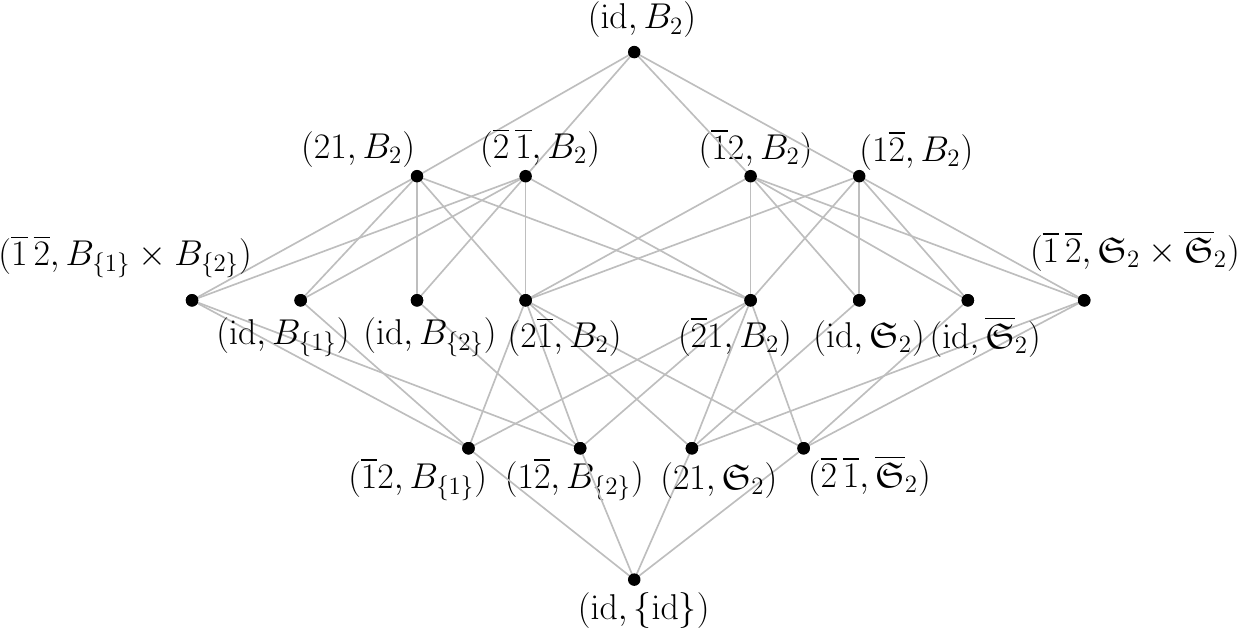}
    \caption{At left, the hyperplane arrangement of $B_2$ labeled by reflections (as signed permutations in one-line notation; in black) and the subgroups they generate (in blue). At right, the poset of pairs $(g_i,W_i)$ as in Section~\ref{sec:poset}, whose maximal chains correspond to minimum-length full reflection factorizations of the identity in $B_2$.  The longest element $-1 = w_0 = \bar{1}\bar{2}$ (which is \emph{not} parabolic quasi-Coxeter in $B_2$) and the two Coxeter elements $\bar{2}1$ and $2\bar{1}$ appear in the middle rank.}
    \label{fig: poset factorizations B_2}
\end{figure}

\subsection{Uniform proofs and generalizations}
\label{sec:uniform}

The proof of our main theorems relies heavily on the classifications of Weyl groups and complex reflection groups.  It would be desirable to give a uniform proof.  Is it possible to use the cut-and-join recurrence (Theorem~\ref{thm:cut and join}) to give a uniform proof of Theorem~\ref{Thm: Weyl case} (the main theorem for Weyl groups)?  
A good test-case for this idea would be to find a new proof of Theorem~\ref{thm:S_n genus 0} using \eqref{eq:LHS cut and join Weyl groups} that avoids Lagrange inversion (as in Goulden and Jackson \cite{GJ97}) and Strehl's complicated series summation formulas \cite{strehl}.

On the other hand, our main Theorem~\ref{thm:main} addresses only well generated groups and parabolic quasi-Coxeter elements. Its statement does not hold more generally because the collection of relative generating sets will be empty outside that setting (Proposition~\ref{Prop: characterization of pqCox}). Still, one can enumerate full reflection factorizations for any element in any reflection group, so we may ask the following question.

\begin{question}\label{Quest_gen_of_thm}
Is there a generalization of Theorem~\ref{thm:main} that holds for non-parabolic quasi-Coxeter elements as well, or (even more) that holds for all elements in all complex reflection groups?
\end{question}

\subsection{Higher genus Hurwitz numbers in Weyl groups}

For a Weyl group $W$, Weyl's formula for the cardinality of $W$ states that 
\[
\#W=n!\cdot\left(\prod_{i=1}^n c_i\right)\cdot I(W),
\]
where $n$ is the rank of $W$ and the $c_i$ are the coefficients in the expansion of the \emph{highest root} of $W$ in terms of the simple roots (see  \cite[\S11-6]{Kane}). We can use this formula to rewrite Theorem~\ref{Thm: Weyl case}; for example, for the identity element $\id\in W$, it implies that
\[
\Ftr_W(\id)= \frac{1}{\#W} \cdot \#\RGS(W)\cdot n!\cdot\left(\prod_{i=1}^n c_i\right) \cdot \ltr(\id)!.
\]
In \cite[Prop.~3.1]{DLM1}, we showed that the exponential generating function $\FFFtr_W(g; z)$ in \eqref{Eq: defn Ftr_W(g;t)} that counts full factorizations of $g$ by their length is given by
\begin{equation} \label{rel F g.f. and Phi}
\FFFtr_W(g; \log X)=\frac{1}{\#W} \cdot \Phi_{W}(g;X) \cdot (X-1)^{\ltr(g)} \cdot \frac{1}{X^{\#\AAA_W}}
\end{equation}
for some monic polynomial $\Phi_{W}(g;X)$ in $X$ of degree $h\cdot n-\ltr(g)$, where $h$ and $\AAA_W$ are respectively the Coxeter number and reflection arrangement of $W$.  A direct comparison implies that, for a Weyl group $W$, the polynomial $\Phi_{W}(\id; X)$ satisfies
\begin{equation}
\Phi_{W}(\id;1)=\#\RGS(W)\cdot n!\cdot \prod_{i=1}^n c_i .
\end{equation}

In our computational experiments we have observed for Weyl groups $W$ that the polynomials $\Phi_{W}(\id;X)$ seem to have positive integer coefficients in $X$. This suggests the following question.

\begin{question} \label{question: postivity phi Weyl groups}
For a Weyl group $W$, is the polynomial $\Phi_{W}(\id;X)$ in $\NN[X]$? If so, is it  possible to realize it as a generating function for some combinatorial statistic? 
\end{question}

In the case of the symmetric group $\Symm_n$, the first author, Chapuy, and Louf have proven (unpublished) that $\Phi_{\Symm_n}(\id; X)$ is the $h$-polynomial of a transportation polytope whose vertices are indexed by certain labeled trees.  However, we do not have any conjectures for similar interpretations in other Weyl groups. 

We have checked that $\Phi_W(\id;X)$ has positive coefficients for the exceptional Weyl groups $W$. For the remaining infinite families, it is at least easy to relate the polynomials $\Phi_{B_n}(\id; X)$ and $\Phi_{D_n}(\id; X)$ with $\Phi_{\Symm_n}(\id; X)$. The equations below follow, after some technical calculations, from \cite[Thm.~4.1]{DLM1}:
 \begin{align} \label{eq:Phi_id_type_B}
        \Phi_{B_n}(\id;X) &=\Phi_{\Symm_{n}}(\id;X^2)\cdot (X+1)^{2n-2}\cdot (1+X+\cdots + X^{n-1})^2,\\
    \Phi_{D_n}(\id;X) &= \frac{1}{(X-1)^2}\left( (X+1)^{2n-2}\Phi_{\Symm_{n}}(\id;X^2) - 2^{2n-2} X^{\binom{n}{2}}\Phi_{\Symm_n}(\id;X)\right).
\end{align}
From \eqref{eq:Phi_id_type_B} one can easily deduce the positivity of $\Phi_{B_n}(\id;X)$, while for $\Phi_{D_n}(\id;X)$ we checked it with a computer for $n\leq 20$.  For non-Weyl types however, the positivity is not true (already for the cyclic group $W=G(6,1,1)$, see \cite[Rem.~3.3]{DLM1}). 

\begin{remark}\label{rem: sum of exponentials and asymptotics}
For a Weyl group $W$ and an element $g \in W$, by \eqref{rel F g.f. and Phi} and since $\Phi_W(g; X)$ is a polynomial in $X$,  the number \[a_W(g;N) := \left[\frac{z^N}{N!}\right] \FFFtr_W(g;z)\] of full reflection factorizations of $g$ of length $N$  is of the form $a_W(g;N)=(\#W)^{-1}\sum_{i \in I} \kappa_i \cdot i^N$ for a finite set $I$ of integers and coefficients $\kappa_i := [X^{i+\#\AAA_W}] \,\Phi_W(g;X) \cdot (X-1)^{\ltr(g)}$.  In particular, the ordinary generating function $\sum_{N \geq 0} a_W(g;N) z^N$ is rational. Asymptotically, the growth of $a_W(g; N)$ is controlled by the largest elements of $I$ (in magnitude); in the Frobenius formula, by \cite[Prop.~3.2]{D2}, this dominant contribution comes from the trivial and sign representations. In particular, as $N \to \infty$ with fixed parity according to $\lR(g)$, we have the asymptotic formula (cf. \cite[Eq.~(3)]{DYZ})
\[
a_W(g;N) \sim \frac{2}{\# W} (\#\RRR)^N.
\]
We warn the reader that this asymptotic formula must be adjusted in non-Weyl types. In particular, it remains true when $W$ is a complex reflection group generated by order-$2$ reflections, but the factor $2/\#W$ changes to $1/\#W$ in the remaining cases (where also there is no parity restriction on $N$).
\end{remark}

\subsection{On the roles of the factors \texorpdfstring{$\Fred(g)$}{FredW(g)} and \texorpdfstring{$\# \RGS(W,g)$}{\# RGS(W, g)}} 
\label{sec:over-counting}

If we rewrite the formula of Theorem~\ref{Thm: Weyl case} as in Remark~\ref{Rem: fred(g_i)->fred(g)}, replacing all the terms $\Fred(g_i)$ by the single term $\Fred(g)$ and writing $k=\lR(g)$ and $n=\rank(W)$, we get the count
\begin{align}
\Ftr_W(g) &= \frac{\ltr(g)!}{\lR(g)!} \cdot \Fred(g) \cdot \#\RGS(W,g)\cdot \frac{I(W_g)}{I(W)}\label{Eq: Weyl better thm}\\
&=\dfrac{(2n-k)!}{n!}\cdot\Bigg[\binom{n}{k}\cdot\Fred(g)\cdot (n-k)!\cdot\#\RGS(W,g)\Bigg]\cdot\dfrac{I(W_g)}{I(W)},\nonumber
\end{align}
where we have used from Theorem~\ref{Prop: characterization of pqCox} that $\ltr(g)=2n-k$ and rearranged some factorials. An immediate observation is that the quantity inside the brackets divides the number of full factorizations $\Ftr_W(g)$ in almost all cases.\footnote{One exception is $W=D_5$ with a parabolic Coxeter element $g$ 
of type $A_4$, where the denominator $I(W) = 4$ is not canceled by the remaining factors.} There is a  natural collection of full factorizations that is enumerated by precisely the bracketed number: it consists of all possible shuffles between a reduced reflection factorization $(t_1,\ldots,t_k)$ of $g$ and a factorization of the form 
\[
\id=t_{k+1}t_{k+1}\cdot t_{k+2}t_{k+2}\cdots t_nt_n,
\]
where $\{t_{k+1},\ldots t_n\}\in\RGS(W,g)$ and where we are not allowed to break the pairs $t_{k+j}t_{k+j}$ in the shuffle. Let us call such full factorizations \defn{simply shuffled}.

These factorizations played an important role in \cite{DLM2} in the proof of Theorem~\ref{Prop: characterization of pqCox}(iv), which characterizes parabolic quasi-Coxeter elements in terms of their full reflection length. They are also essentially the types of factorizations that appear in \cite[Cor.~1.4]{LR}, which tells us that every minimum-length full reflection factorization of $g$ is Hurwitz-equivalent to a simply shuffled one. As it happens, all minimum-length full factorizations of a fixed parabolic quasi-Coxeter element are Hurwitz-equivalent \cite[Prop.~4.1]{DL}, so that we may hope for an explanation of Theorem~\ref{Thm: Weyl case} along the following lines.

\begin{question}\label{Quest:overcounting_braids}
Using the notation of the previous discussion, is there a collection of $I(W_g)\cdot (2n-k)!/n!$-many elements in the braid group $\BBB_{2n-k}$ that produce, via their Hurwitz action on the simply shuffled factorizations, \emph{all} minimum-length full reflection factorizations of $g$, each of them $I(W)$-many times?
\end{question}

It is not unreasonable to expect $I(W)$ to be an overcounting factor: for example, in \cite[\S5]{alcovedII}, the authors construct a finite subgroup $C\leq W$ with cardinality equal to the connection index, and it is not difficult to see that this subgroup acts freely on ordered good generating sets.

One approach towards Question~\ref{Quest:overcounting_braids} could be the recent work \cite{WY2} of Wegener and Yahiatene. In \cite[Rem.~4.2]{WY2} they give an explicit algorithm that can transform any full factorization to a simply shuffled one. Perhaps by studying the fiber of the Wegener--Yahiatene construction one could explain the numerological properties discussed above.

\begin{remark}
Alternatively, it seems that $\Fred(g)$ divides $\Ftr_W(g)$ for all parabolic quasi-Coxeter elements $g$ in every Weyl group $W$.  Can this be explained directly?
\end{remark}

Finally, we remark that trying to replace $\Fred(g)$ with $\#\RGS(W,g)$ as the main object in a combinatorial proof seems less promising. In particular, the number of subsets of the terms in a minimum-length full reflection factorization of $g$ that belong to $\RGS(W,g)$ is not constant, already in $\Symm_n$ and for the identity $g= \id$.  
For example, take the two factorizations 
\[
(12)(12)(23)(23)=\id\qquad\text{ and }\qquad (12)(23)(13)(23)=\id.
\] 
The first has four two-element subsets that generate the whole group $\Symm_3$ and two that don't, while the second one has five subsets that generate $\Symm_3$ and only one that doesn't.  This makes it more complicated to find some sort of overcounting argument to explain why the number of generating sets appears in the formula.  Is there one even just for $\Symm_n$?

\subsection{Possible geometry behind \texorpdfstring{$W$}{W}-Hurwitz numbers}
\label{sec:geometry}

In the symmetric group, the genus-$g$ Hurwitz numbers $H_g(\lambda)$, where $\lambda=(\lambda_1,\ldots,\lambda_r)$ can be written as 
\[
H_g(\lambda) \,=\, (n+r+2g-2)!\cdot 
P_{g,r}(\lambda) \cdot  \prod_{i=1}^r \frac{\lambda_i^{\lambda_i}}{(\lambda_i-1)!},
\]
where $P_{g,r}(\lambda)$ is a polynomial in $\lambda_1,\ldots,\lambda_r$ \cite[Eq.~(1.1)]{GJVain}. For instance, in the case $g = 0$, one has $P_{0,r}(\lambda)=(\lambda_1+\cdots + \lambda_r)^{r-3}$, recovering \eqref{EQ: Hurwitz formula}. This polynomiality was conjectured by Goulden, Jackson and Vakil \cite{GJV01} and proved by Ekedahl, Lando, Shapiro and Vainshtein \cite{ELSV}, who found an expression, called the \defn{ELSV formula}, for the polynomial $P_{g, r}$ as a {\em Hodge integral}:
\begin{equation} \label{eq:ELSV}
P_{g,r}(\lambda) =\int_{\overline{\mathcal{M}}_{g,r}} \frac{C(E^*)}{(1-\lambda_1 \psi_1)\cdots (1-\lambda_r \psi_r)} \in \QQ[\lambda_1,\ldots,\lambda_r],
\end{equation}
where $\overline{\mathcal{M}}_{g,r}$ is the compact moduli space of stable $r$-pointed genus $g$ curves and $C(E^*)$, $\psi_i$ are certain Chern classes (see \cite[\S5.3]{LandoZvonkin} or \cite{lando_survey} for details). 

It would be of interest to find an analogue of the ELSV formula for full factorizations in complex reflection groups; indeed, this was the motivation of the work of Pollack--Ross in \cite{PR}. A full generalization is currently out of reach because there is no known analogue for the moduli space $\overline{\mathcal{M}}_{g,r}$ for other reflection groups, nor is it clear what the functions on such a moduli space should look like.

However, a part of this story does extend. On the sphere (the genus-$0$ case), a holomorphic function of degree $r$ with a single pole of order $r$ is just a complex polynomial of degree $r$. This is also the general deformation of the type-$A_{r-1}$ simple singularity $z^r=0$. One of the important ingredients of the ELSV formula is that the \emph{Lyashko--Looijenga (LL) morphism} (see \cite{looijenga,Bessis_Annals,DLM2}) extends algebraically to a cone over $\overline{\mathcal{M}}_{g,r}$ and that it is compatible  with the $\CC^\times$ action on the fibers of the cone (i.e., it is {\em weighted-homogeneous}); see \cite[\S4.6]{lando_survey} for details. Originally, the LL map was defined for all simple singularities (types $A$, $D$, $E$) and it is in all cases weighted-homogeneous on their semiuniversal deformation spaces. The monodromies of these simple singularities naturally correspond to the Coxeter elements of the corresponding reflection groups and the degree of the corresponding LL map counts minimum-length reflection factorizations of a Coxeter element (see \cite{looijenga, Bessis_Annals, D1}). 

We also note that in the genus-$0$ case, the space $\overline{\mathcal{M}}_{0,r}$  is isomorphic to the {\em wonderful compactification} of the braid arrangement \cite[\S4.3]{dCP_wonder}. The other reflection groups also have  wonderful compactifications \cite{dCP_holon}. It would be interesting to define a branching map on such compactifications, maybe building on the geometric interpretation of \cite{AFV_gaudin}.

\subsection{Root systems for complex  reflection groups}
\label{sec: michel root systems}

A central object in our Theorem~\ref{thm:main} is a statistic for each (relative) generating set: the Grammian determinant of the corresponding set of roots. For a Weyl group $W$ and a good generating set of $W$, the Grammian determinant agrees with the connection index $I(W)$ (see Section~\ref{sec: equivalence of main thms for Weyl groups}). In \cite{BCM}, Broue--Corran--Michel developed a theory of (co)root systems for arbitrary complex reflection groups, including an analogue of a connection index. Given a well generated group $W$, their object is the ideal that is generated by the determinant of a Cartan matrix associated to $W$ \cite[Def.~3.47, Prop.~6.7]{BCM}.  This determinant  is precisely the Grammian determinant of the generating set of reflections of $W$ that gives the Cartan matrix. Different generating sets would give different Grammian determinants, and this is important for Theorem~\ref{thm:main}, but they all determine the same ideal in the setting of \cite{BCM}.  Thus in some sense, the statistic we study is not just a main ingredient of our Theorem~\ref{thm:main}, but is a meaningful representation-theoretic object on its own. Can this perspective help towards a more uniform understanding of Theorem~\ref{thm:main}?

\subsection{Relation to weighted Hurwitz numbers}

In Sections~\ref{sec: proof Gm1n} and \ref{sec: proof Gmmn}, as part of the proof of the main theorem for the infinite families, we used formulas \eqref{eq:case 1(a) LHS}, \eqref{eq:case 1(b) LHS}, \eqref{eq:case 2(a) LHS}, \eqref{eq:case 2(b) LHS} that might be viewed as weighted versions of the Hurwitz numbers of $\Symm_n$. Putting weights on Hurwitz numbers of $\Symm_n$ is a rich and active area of study. For example, Harnad and collaborators \cite{G-PH_WHN,A_C_E_H,A_C_E_H2,bonzom2021b} have studied different {\em weighted Hurwitz numbers} determined by a weight generating function.  This approach unifies {\em double Hurwitz numbers} \cite{okounkov_toda} and {\em monotone Hurwitz numbers} \cite{G_G-P_N}, among other variations---see \cite{HarWHN} for a survey. It would be interesting to see if the weighted Hurwitz numbers in our proofs are related.

\section*{Acknowledgements}

We thank Guillaume Chapuy, David Jackson, Jean Michel, Gilles Schaeffer, Ravi Vakil, and Jiayuan Wang for helpful comments and suggestions.
This work was facilitated by computer experiments using Sage \cite{sagemath}, its algebraic combinatorics features developed by the Sage-Combinat community \cite{Sage-Combinat}, and CHEVIE \cite{chevie}. We thank the staff of the RCF computing facility at the Math and Stats Department at UMass Amherst and especially Rachel Aronow for their support using the RCF cluster. 

An extended abstract of this work appeared as \cite{DLM-extended-abstract}.

\appendix

\section{Proofs of identities involving roots of unity}
\label{appendix}

We begin with the proof of the identity used in the proof of  Theorem~\ref{thm:main} for $G(m,1,n)$.

\begin{proposition}
\label{prop. primitive root identity 1}
For a nonnegative integer $m>1$ we have that 
\begin{equation}
\label{eq: primitive root identity 1}
\sum_{\xi \text{ prim.}} \frac{1}{1-\xi} = \frac{\varphi(m)}{2},
\end{equation}
where the sum on the left side is over primitive $m$th roots of unity.
\end{proposition}

\begin{proof}
By M\"obius inversion it suffices to show that 
\[
a_m:=\sum_{\xi } \frac{1}{1-\xi} = \frac{m-1}{2},
\]
where the sum on the left side is over all nontrivial $m$th roots of unity. 
We have that $a_m = f'(1)/f(1)$ where $f(x) := (x^m-1)/(x-1)= x^{m-1}+x^{m-2}+\cdots+1$. Since $f(x)$ is palindromic, it follows that $f'(1)/f(1)$ equals half its degree, and the result follows.  
\end{proof}

\begin{remark}
Alternatively, the quantity on the left side of \eqref{eq: primitive root identity 1} is equal to $\Phi'_m(1)/\Phi_m(1)$, where $\Phi_m(x)$ is the $m$th cyclotomic polynomial. Since this polynomial is also palindromic, then
$\Phi_m'(1)/\Phi_m(1)$ equals half its degree, which is $\varphi(m)/2$.
\end{remark}

Next we give the proof of the identity used in the proof of  Theorem~\ref{thm:main} for $G(m,m,n)$. As a first step, we give a technical lemma on \defn{Chebyshev polynomials of the first kind}, denoted $T_n(x)$ for $n \geq 0$.  For our purposes, it is convenient to take as the definition of Chebyshev polynomials their generating function \cite[Eq.~1.105]{rivlin2020chebyshev}:
\begin{equation} \label{eq:gf T}
\sum_{n=0}^{\infty} T_n(x) z^n = \frac{1-xz}{1+x^2-2xz}.
\end{equation}

\begin{proposition} \label{prop. helper identities Chebyshev}
Fix a positive integer $s$.  Let \[
a(x):=\frac{T_{s+1}(x)-T_s(x)}{x-1} 
\qquad \text{and} \qquad b(x):=\frac{T_{s+1}(x)-T_{s-1}(x)}{x-1}. 
\] 
Then $a(x)$ and $b(x)$ are polynomials such that
\begin{enumerate}[(a)]
\item $a(1) = 2s+1$ and $a'(1) = s(s+1)(2s+1)/3$, and
\item $b(1)=4s$ and $b'(1)=2s(2s^2+1)/3$.
\end{enumerate}
\end{proposition}

\begin{proof}
We prove part (a); the proof for (b) is analogous.

By setting $x = 1$ in \eqref{eq:gf T}, we see
$T_n(1)=1$ for all $n$.  Then it follows that $A(x) := T_{s+1}(x)-T_s(x) = (x-1)a(x)$ is divisible (as a polynomial) by $x - 1$, so $a(x)$ is a polynomial. By the
product rule, we have that $a(1)=A'(1)$ and $a'(1)=A''(1)/2$.

By differentiating the generating
function \eqref{eq:gf T} with respect to $x$ and evaluating at $x=1$ we obtain
\[
\sum_{n=0}^{\infty} T'_n(1) z^n = \frac{z^2+z}{(1-z)^3},
\]
which implies that $T_n'(1)=n^2$ and so $a(1) = A'(1) =  (s+1)^2 - s^2 = 2s+1$.

Next we show the identity for $a'(1)$. By differentiating the generating
function \eqref{eq:gf T} twice with respect to $x$ and evaluating at $x=1$ we obtain
\[
\sum_{n=0}^{\infty} T''_n(1) z^n = \frac{4z^3+4z^2}{(1-z)^5},
\]
which implies that $T_n''(1)=n^2(n^2-1)/3$ and so
\[
a'(1) = \frac{A''(1)}{2} = \frac{(s+1)^2((s+1)^2-1) - s^2(s^2-1)}{6} =  \frac{s(s+1)(2s+1)}{3}.
\qedhere
\]
\end{proof}

\begin{proposition} \label{prop: primitive root identity 2}
For a nonnegative integer $m\geq 2$ we have that 
\begin{equation} \label{eq: primitive root identity 2}
\sum_{\xi \text{ prim.}} \frac{1}{2 - \xi - \overline{\xi}} = \frac{1}{12}\sum_{r\mid m} r^2 \mu(m/r),
\end{equation}
where the sum on the left side is over primitive $m$th roots of unity.
\end{proposition}

\begin{proof}
For $m = 2$, the result is a trivial calculation.  So assume $m > 2$.  Let $s = \left\lfloor \frac{m}{2} \right\rfloor$.
By M\"obius inversion, it is enough to show that
  \[
c_m := \sum_{\substack{\xi \colon \xi^m = 1, \\ \xi \neq 1}} \frac{1}{2 - \xi - \overline{\xi}} = \frac{m^2-1}{12},
\]
where the sum is over all nontrivial $m$th roots of unity.  When $m$ is odd, each summand appears twice (once for $\xi$ and once for $\overline{\xi}$), while when $m$ is even, each summand appears twice except for the root $\xi=-1 = \overline{\xi}$ (see Figure~\ref{fig: identity even odd cases roots}). By combining like terms, we can rewrite $c_m$ as 
\[
c_m = \sum_{k=1}^s \frac{1}{1-\cos(2\pi k/m)} \,-\, \begin{cases} 
      0 &\text{if $m$ is odd},\\
      1/4 & \text{if $m$ is even}.
\end{cases}
\]
Let $\Psi_m(x)$ be the minimal polynomial over $\QQ$ of $\cos(2\pi/m)$, and let $Q_m(x) := \prod_{r \mid m, r\neq 1} \Psi_r(x)$.  It follows from the analysis of these polynomials in \cite{WZ} that 
\[
Q_m(x) = \prod_{k = 1}^{s} \left(x - \cos(2\pi k/ m)\right),
\]
and consequently that
\[
    c_m = \frac{Q'_m(1)}{Q_m(1)} - \begin{cases} 
      0 &\text{if $m$ is odd},\\
      1/4 & \text{if $m$ is even}.
      \end{cases}
\]
On the other hand, the main theorem of \cite{WZ} directly implies that
  \[
Q_m(x) =  \frac{1}{2^s (x - 1)} \cdot \begin{cases}
  T_{s+1}(x) - T_s(x)     &\text{if $m$ is odd},\\
  T_{s+1}(x) - T_{s-1}(x) &\text{if $m$ is even}.
    \end{cases}
  \]

When $m = 2s + 1$ is odd, we have by Proposition~\ref{prop. helper identities Chebyshev}(a) that $Q_m(1)=(2s+1)/2^{s}$ and $Q'_m(1)=s(s+1)(2s+1)/(3 \cdot 2^s)$, and so
\[
 \frac{Q'_m(1)}{Q_m(1)} = \frac{s(s+1)}{3} = \frac{m^2 - 1}{12}.
\]
When $m=2s$ is even, we have by Proposition~\ref{prop. helper identities Chebyshev}(b) that $Q_m(1) = 4s/2^s$ and
$Q'_m(1)=2s(2s^2+1)/(3\cdot 2^s)$
and thus
\[
\frac{Q'_m(1)}{Q_m(1)} -\frac{1}{4} = \frac{2s^2 + 1}{6} - \frac{1}{4} = \frac{m^2-1}{12}.
\]
This completes the proof.
\end{proof}

\begin{remark}\label{Rem: Dedekind}
The sequence  $J_2(m) = \sum_{r\mid m} r^2 \mu(m/r)$ appearing on the right side of \eqref{eq: primitive root identity 2} is \href{https://oeis.org/A007434}{A007434} in \cite{oeis}; it counts elements of order $m$ in the bicyclic group $(\ZZ/m\ZZ)^2$.  It is an integer multiple of Euler's totient $\varphi(m)$, with quotient given by the \emph{Dedekind $\psi$-function} \cite[\href{https://oeis.org/A001615}{A001615}]{oeis}.
\end{remark}

\begin{figure}
    \centering
    \includegraphics{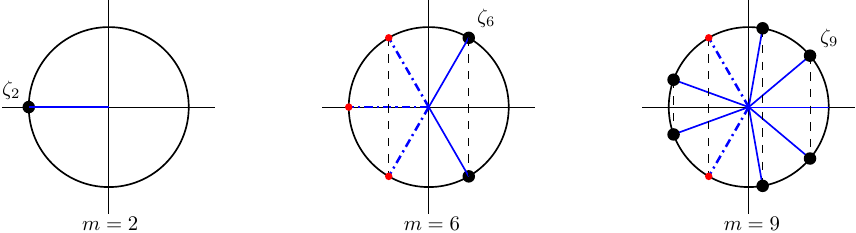}
    \caption{Illustration of the quantities $\xi+\overline{\xi}$ over roots of unity appearing in the identity of Proposition~\ref{prop: primitive root identity 2} for $m=2, 6, 9$. The primitive roots are denoted in black ($\bullet$) while the other roots are denoted in red (\textcolor{red}{$\bullet$}). The root $\xi$ is paired with $\overline{\xi}$ by dashed lines.}
    \label{fig: identity even odd cases roots}
\end{figure}

\printbibliography

\end{document}